\RequirePackage[l2tabu, orthodox]{nag}
\documentclass[11pt]{article}

\usepackage[utf8]{inputenc}
\usepackage[T1]{fontenc}
\usepackage{lmodern}
\usepackage[british]{babel}
\usepackage{csquotes}
\usepackage{microtype}
\usepackage[a4paper]{geometry}
\usepackage{hyperref}
\usepackage{graphicx}  
\usepackage{subcaption} 
\usepackage{multirow}
\usepackage{booktabs}
\usepackage{pgfplots}
\usepackage{tikz}
\usepackage{tikzscale}
\usepackage{amsmath}
\usepackage{amssymb}
\usepackage{amsthm}
\usepackage{mathtools}
\usepackage{siunitx}
\usepackage{dsfont}
\usepackage{authblk}
\usepackage{xifthen}
\usepackage[section]{placeins} % keep floats within sections
\usepackage{mleftright}
\usepackage{fancyhdr}

\definecolor{TolDarkPurple}{HTML}{332288}
\definecolor{TolDarkBlue}{HTML}{6699CC}
\definecolor{TolLightBlue}{HTML}{88CCEE}
\definecolor{TolLightGreen}{HTML}{44AA99}
\definecolor{TolDarkGreen}{HTML}{117733}
\definecolor{TolDarkBrown}{HTML}{999933}
\definecolor{TolLightBrown}{HTML}{DDCC77}
\definecolor{TolDarkRed}{HTML}{661100}
\definecolor{TolLightRed}{HTML}{CC6677}
\definecolor{TolLightPink}{HTML}{AA4466}
\definecolor{TolDarkPink}{HTML}{882255}
\definecolor{TolLightPurple}{HTML}{AA4499}

\pgfplotsset{compat=newest}
\mathtoolsset{showonlyrefs} 
\numberwithin{equation}{section}
\pgfkeys{/pgf/number format/.cd,1000 sep={\thinspace}} % thousand separator
\hypersetup{colorlinks=true,
  linkcolor=TolDarkBlue,
  citecolor=TolLightGreen,
  urlcolor=TolDarkRed}
\definecolor{InfernoMax}{rgb}{0.988362,0.998364,0.644924}
\usepackage[backend=biber,
hyperref=auto,
style=numeric,
citestyle=numeric-comp,
sorting=nyt,
block=space,
maxbibnames=9,
maxcitenames=2
]{biblatex}
\DeclareMathOperator*{\diag}{\ensuremath{\mathrm{diag}}}
\DeclareMathOperator*{\vspan}{\ensuremath{\mathrm{span}}}
\DeclareMathOperator*{\vdim}{\ensuremath{\mathrm{dim}}}
\DeclareMathOperator*{\rank}{\ensuremath{\mathrm{rank}}}
\DeclarePairedDelimiter{\abs}{\lvert}{\rvert}
\DeclarePairedDelimiter{\norm}{\lVert}{\rVert}

\makeatletter
\newcommand{\myie}{i.e\@ifnextchar.{}{.\@}}
\newcommand{\myIe}{I.e\@ifnextchar.{}{.\@}}
\newcommand{\myeg}{e.g\@ifnextchar.{}{.\@}}
\newcommand{\myviz}{viz\@ifnextchar.{}{.\@}}
\newcommand{\mycf}{cf\@ifnextchar.{}{.\@}}
\makeatother

\theoremstyle{plain}
\newtheorem{prop}{Proposition}
\theoremstyle{definition}
\newtheorem{hyp}[prop]{Assumption}
\theoremstyle{remark}
\newtheorem{rmk}[prop]{Remark}

\title{\bfseries Tensor-based multiscale method for diffusion problems in quasi-periodic heterogeneous media}
\author[1]{Quentin Ayoul-Guilmard\thanks{quentin.ayoul-guilmard@centraliens-nantes.net}}
\author[2]{Anthony Nouy\thanks{anthony.nouy@ec-nantes.fr}}
\author[1]{Christophe Binetruy}
\affil[1]{\'Ecole Centrale de Nantes, GeM UMR CNRS 6183, Nantes, France}
\affil[2]{\'Ecole Centrale de Nantes, LMJL UMR CNRS 6629, Nantes, France}

\date{}

\begin{document}%*****************************************************************

\maketitle

\begin{abstract}
  This paper proposes to address the issue of complexity reduction for the numerical simulation of multiscale media in a quasi-periodic setting.
  We consider a stationary elliptic diffusion equation defined on a domain $D$ such that $\overline{D}$ is the union of cells $\{\overline{D_i}\}_{i\in I}$ and we introduce a two-scale representation by identifying any function $v(x)$ defined on $D$ with a bi-variate function $v(i,y)$, where $i \in I$ relates to the index of the cell containing the point $x$ and $y \in Y$ relates to a local coordinate in a reference cell $Y$.
  We introduce a weak formulation of the problem in a broken Sobolev space $V(D)$ using a discontinuous Galerkin framework.
  The problem is then interpreted as a tensor-structured equation by identifying $V(D)$ with a tensor product space $\mathbb{R}^I \otimes V(Y)$ of functions defined over the product set $I\times Y$.
  Tensor numerical methods are then used in order to exploit approximability properties of quasi-periodic solutions by low-rank tensors.
  
  \bigskip{}\noindent\textbf{2010 Mathematics subject classification:} 15A69, 35B15, 65N30.

  \medskip{}\noindent\textbf{Keywords:} quasi-periodicity, tensor approximation, discontinuous Galerkin, multiscale, heterogeneous diffusion.

\end{abstract}

\section{Introduction}

Heterogeneous periodic media are increasingly common in the industry, particularly owing to the use of architectured microstructure (\myeg{} composite materials).
Their complex behaviour calls for thorough and expensive experimental investigations.
As an alternative, numerical simulations involve fine-scale models which often require heavy computations.
Periodicity assumption on the medium means that all its information is contained within a single cell, which can be exploited in practical resolutions (\myeg{} homogenisation).
Nonetheless, the need to withdraw this assumption arises with situations such as defect impact studies; this raises a computational challenge.

To the best authors' knowledge, there exist currently two families of approaches available to tackle such problem more efficiently than brute fine-scale computation---such as typical finite element method.
First is the set of multiscale methods such as Multiscale Finite Element Method (MsFEM)~\cite{Efendiev2009a,Allaire2005,Hou1997,Bal2011}, Heterogeneous Multiscale Method (HMM)~\cite{Abdulle2012b,Engquist2007,Bal2011} or patch methods~\cite{Lions1999a,Glowinski2003,Rezzonico2007,Gendre2011a,Chevreuil2013b}.
Although these are designed to address the issue of multiscale complexity, they are intended for broader purposes than our particular case of interest; as such, they fail to achieve the complexity reduction one could expect from a quasi-periodicity assumption.
Secondly, progress has been made over the past few years toward exploitation of quasi-periodicity in stochastic homogenisation methods.
These works focus on computational cost reduction of classical stochastic homogenisation through suitable assumption on the stochastic model~\cite{LeBris2009,Blanc2006,Anantharaman2011a}, as well as specific variance reduction schemes~\cite{Blanc2012a,Legoll2015a,Legoll2015b} and an adaptation to special quasirandom structures used in atomistic simulations~\cite{LeBris2016}.
The aforementioned methods exploit quasi-periodicity in order to reduce the number of supercell problems to solve, comparatively to classical stochastic homogenisation.
Consequently, they are cost-efficient to compute good approximations of homogenised quantities of a material ideally periodic yet perturbed by random imperfections.
They do not, however, reduce complexity of a given deterministic, quasi-periodic supercell problem such as those they involve.
To address this computational bottleneck, various adaptations of aforementioned general multiscale methods have been developed (\myeg{}~\cite{LeBris2014a}).
Several noteworthy approaches based on reduced basis methods (whose principle is explained in~\cite{Maday2006c}) have been developed to exploit quasi-periodic patterns, such as~\cite{Abdulle2012c,Boyaval2008a,LeBris2012a}.
We propose here a multiscale method designed specifically to address such quasi-periodic problems.

Section~\ref{sec:disc-galerk-form} will introduce the reference problem, a two-scale representation and the related discontinuous Galerkin formulation.
In section~\ref{sec:tensor}, we identify the problem as an operator equation in a Hilbert tensor space and we use a greedy algorithm for the construction of a sequence of low-rank approximations of the solution.
Finally, section~\ref{sec:numerical-results} illustrates the efficiency of the proposed method through a number of representative numerical experiments.

\section{Reference problem and discontinuous Galerkin formulation}
\label{sec:disc-galerk-form}

Let $D\subset\mathbb{R}^d$ be an open rectangular cuboid.
We consider a stationary diffusion equation 
\begin{gather}
  - \nabla\cdot(K\nabla u) = f \quad \text{ in } D \label{eq:diffusion-strong},
\end{gather}
with periodic boundary conditions, where $K$ is the diffusion (or \enquote{conductivity}) coefficient and $f$ is a source term. An example of quasi-periodic heterogeneous two-phase material is given on figure~\ref{fig:cookies}.

\begin{figure}
  \centering
  \includegraphics{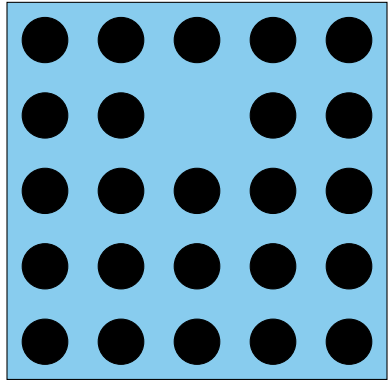}
  \caption{Periodic medium with one defect}
  \label{fig:cookies}
\end{figure}

We assume that $K\in L^\infty(D)$ and $f\in L^2(D)$. A weak solution
\begin{gather}
  u\in H^1_{per}(D)= \mleft\{v\in H^1(D) : v~D\text{-periodic}\mright\}
\end{gather}
of~\eqref{eq:diffusion-strong} is such that 
\begin{gather}
  \label{eq:diffusion-continuous}
  \int_D K\nabla u \cdot\nabla v = \int_Dfv, \quad \forall v\in H^1_{per}(D).
\end{gather}

\subsection{Mesoscopic discretisation}
\label{sec:discon-setting}

We introduce a partition of $\overline{D}$ into closed domains $\{\overline{D_i}\}_{i\in I}$, where $I$ is totally ordered set.
The subsets $(D_i)_{i\in I}$ are open and identical up to a translation.
They will be called \enquote{cells} and are chosen so as to fit the quasi-periodically repeated patterns (see figure~\ref{fig:cookies-grid}).

\begin{figure}
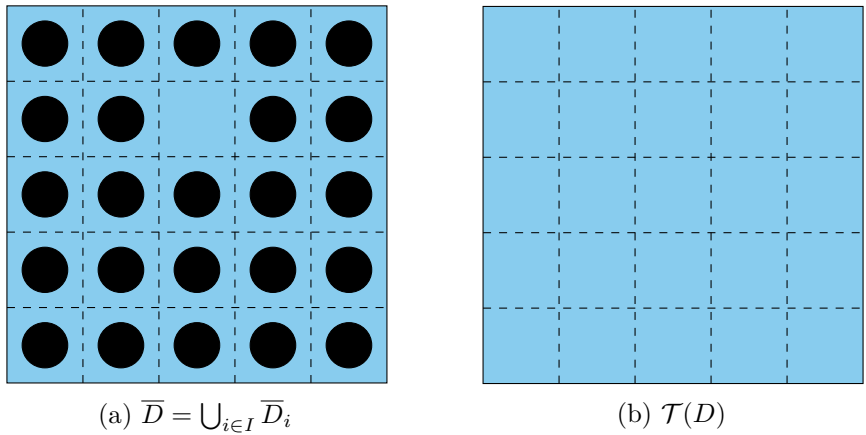

  \centering{}
  \begin{subfigure}{.4\linewidth}
    \centering
    \includegraphics{cookies_meso-grid.tikz}
    \caption{$\overline D=\bigcup_{i\in I} \overline D_i$}
    \label{fig:cookies-grid}
  \end{subfigure}
  ~
  \begin{subfigure}{.4\linewidth}
    \centering
    \includegraphics{meso-mesh.tikz}
    \caption{$\mathcal{T}(D)$}
    \label{fig:meso-mesh}
  \end{subfigure}
  \caption{Mesoscopic mesh $\mathcal{T}(D)$ of domain $D$}
\end{figure}

The set of cells defines a mesoscopic mesh $ \mathcal{T}(D)=\{D_i:i\in I\}$ over $D$ (see figure~\ref{fig:meso-mesh}).
We denote by $\mathcal{F}(D)$ the set of \emph{faces}, by $ \mathcal{F}_e(D) =\{F\in\mathcal{F}(D)~:~F\subset\partial D\}$ the set of \emph{external faces}, and by $\mathcal{F}_i(D) = \mathcal{F}(D)\setminus\mathcal{F}_e(D)$ the set of \emph{internal faces}.
We define in the same way $\mathcal{F}(D_i)$, the set of faces of cell $D_i$ for any $i\in I$.

For a face $F\in \mathcal{F}_i{}(D)$, we let $(i,j) \in I^2$ be the unique ordered pair of indices such that $F = \partial D_i\cap \partial D_j$.
We denote by $n_F$ the unit normal vector of face $F$, outward of $D_i$.
For a function $v$ defined over $D_i\cup D_j$, we denote---if it exists---the trace $(v_{|D_i})_{|F}$ on $F$ of its restriction $v_{|D_i}$.
We then define the average operator over face $F$ $\{\cdot\}_F$ by $\{v\}_F = \frac{1}{2}((v_{|D_i})_{|F}+(v_{|D_j})_{|F})$, and the jump operator over face $F$ $[\cdot]_F$ by $[v]_F = ((v_{|D_i})_{|F}-(v_{|D_i})_{|F})$.

For the sake of simplicity, the subscript $F$ will be omitted whenever the face related to is obvious.
Periodic boundary conditions allow to extend these definitions to external faces by identifying a face $F = \partial D_i \cap \partial D\in\mathcal{F}_e(D)$ with the opposite face $F'=\partial D_j \cap \partial D\in\mathcal{F}_e(D)$, which we will use below (see remark~\ref{rmk:pbc-faces}).
For the definition of the normal $n_F$ and the jump operator, we use again the convention $i<j$.

We then introduce the broken Sobolev space
\begin{gather}
  H^1\mleft(\bigcup_{i\in I} D_i\mright)= \mleft\{v \in L^2(\mathbb{R}) : \forall i\in I, v_{|D_i}\in H^1(D_i) \mright\}.
\end{gather}
It should be noted that, since the cells $D_i$ are open, $\bigcup_{i\in I} D_i\neq D$.

\subsection{Symmetric weighted interior penalty (SWIP) formulation}
\label{sec:symm-inter-penalty}
We make the following assumption on the regularity of the solution.

\begin{hyp}
  \label{hyp:regularity}
  We assume that the solution $u$ of~\eqref{eq:diffusion-strong} is in $H^1_{per}(D)\cap H^2(D)$ so that, for all $ F\in\mathcal{F}(D)$,
  \begin{gather}
    \int_F [u] = 0  \quad \text{and} \quad  \int_F [K\nabla u]\cdot n = 0.
  \end{gather}
\end{hyp}
From~\textcite[th. 1 § 6.3.1 p. 309]{Evans1998a}, if $-\nabla\cdot(K\nabla u)=f$ with $K\in C^1(D)$ and $f\in L^2(D)$, then $u\in H^2_{loc}(D)$ whatever the boundary conditions.
If $u$ is $D$-periodic, then $u\in H^2(D)$ since
\begin{align*}
  \mleft\{v\in H^2_{loc}(D) : v \text{ is $D$-periodic}\mright\} = \mleft\{v\in H^2(D) : v \text{ is $D$-periodic}\mright\},
\end{align*}
Therefore, assumption~\ref{hyp:regularity} is verified in our case if $K\in C^1(D)$.
For Dirichlet and Neumann boundary conditions, we refer the reader to~\textcite[th. 3.12 p. 119]{Ern2004}.

For the discontinuous Galerkin formulation to come, we introduce a subset of $H^1(\bigcup_{i\in I}D_i)$ defined as
\begin{gather}
  V(D) := \mleft\{v\in H^1\mleft(\bigcup_{i\in I}D_i\mright) : \forall i \in I, (\nabla v_{|D_i})_{|\partial D_i} \in L^2(\partial D_i)^d \mright\}.
\end{gather}
Then the solution $u$  of~\eqref{eq:diffusion-strong} satisfies
\begin{gather}
  \forall v \in V(D), \quad a(u,v) - c(u,v) = b(v) \label{eq:dG-consistent},
\end{gather}
where $a$ and $c$ are bilinear forms over $V(D)$ respectively defined by
\begin{gather}
  a(u,v) = \sum_{i\in I}\int_{D_i} K\nabla u \cdot\nabla v \label{eq:dG-a}, \quad  c(u,v)  = \sum_{F\in \mathcal{F}(D)} \int_F n\cdot\{K\nabla u\}[v],
\end{gather}
and $b$ is a linear form defined by
\begin{gather}
  b(v) = \int_D fv.
\end{gather}
From assumption~\ref{hyp:regularity} and from the $D$-periodicity of $u$, we have that $u$ also satisfies
\begin{gather}
  \label{eq:dG-SIP}
  \forall v\in V(D), \quad a(u,v) - c(u,v) - c(v,u) + \sum_{F\in\mathcal{F}(D)}\frac{\sigma}{\abs{F}}\int_F [u][v] = b(v),
\end{gather}
with $\abs{F}$ the measure of face $F$, and with $\sigma $ a positive penalty parameter.
Equation \eqref{eq:dG-SIP} corresponds to the symmetric interior penalty (SIP) formulation of~\eqref{eq:diffusion-strong} (see \cite{DiPietro2011} for a detailed explanation), which involves a coercive bilinear form for a sufficiently high value of the penalisation parameter $\sigma$.

In the present context, $K$ may show strong heterogeneities.
The symmetric weighted interior penalty (SWIP) method \cite{DiPietro2011}, a variant of SIP, is designed to account for this by introducing weights in the definition of averages on faces and in the penalty term.
For a cell $D_i$, we let $ k_i^+$ and $k_i ^-$ be the constants defined by
\begin{gather}
  \label{eq:def-K-bounds}
  k_i^- = \inf_{x \in D_i} \lambda_{min}(K(x))\quad  \text{and} \quad k_i^+ = \sup_{x \in D_i} \lambda_{max}(K(x)),
\end{gather}
where $\lambda_{min}(A)$ and $\lambda_{max}(A)$ respectively denote the minimum and maximum eigenvalue of a symmetric matrix $A$.
For a face $F=\partial D_i\cap \partial D_j\in\mathcal{F}(D)$, we define a stabilisation weight $\omega_F$ and average weights $\beta_F^- $ and $\beta_F^+$ as
\begin{gather}
  \label{eq:def-weights}
  \omega_{F} = \frac{2k^+_ik^+_j}{k^+_i+k^+_j},  \quad 
  \beta_{F}^-  = \frac{k_i^+}{k_i^++k_j^+}, \quad  \beta_{F}^+  = \frac{k_j^+}{k_i^++k_j^+}.
\end{gather}
Then, we redefine the average operator $ \{\cdot\}_{F} $ over $F$ by
\begin{gather}
  \label{eq:weighted-avg}
  \{v\}_{F} = \beta_{F}^- (v_{|D_i})_{|F} + \beta_{F}^+(v_{|D_j})_{|F},
\end{gather}
and we introduce a stabilisation bilinear form
\begin{gather}
  \label{eq:stab-form}
  s(v,w) =  \sum_{F\in \mathcal{F}(D)} \sigma\frac{\omega_F}{\abs{F}}\int_F [w]_F[v]_F.
\end{gather}

The problem with periodic boundary conditions admits infinitely many solutions that differ by a constant.
We decide to fix this constant by choosing a particular solution in the kernel of the linear form $\phi(v) = \int_D v$.
This is achieved by introducing a symmetric bilinear form
\begin{gather}
  m(u,v) = \phi(u)\phi(v)
\end{gather}
whose left kernel is the kernel of $\phi$.

Finally, we achieve a consistent SWIP formulation, \myie{} the solution $u\in H^1_{per}(D)\cap H^2(D)$ of~\eqref{eq:diffusion-strong} verifies
\begin{gather}
  \label{eq:dG-SWIP-D}
  \forall v\in V(D), \quad a^{swip}(u,v) = b(v),
\end{gather}
with $a^{swip}(u,v) = a(u,v) - c(u,v) - c(v,u) + s(u,v) + m(u,v)$.

\begin{rmk}[Periodic boundary conditions' enforcement]
  \label{rmk:pbc-faces}
  As explained in section~\ref{sec:discon-setting}, periodic boundary conditions give meaning to an extension of face jump and face average operators to external faces.
  As far as $D$-periodic functions are concerned, these external faces can be considered as internal faces.
  Thus, in formulation~\eqref{eq:dG-SWIP-D}, periodic boundary conditions are weakly enforced through the terms $\frac{\sigma\omega_F}{\abs{F}}\int_F [u][v]$ associated with faces $F\in \mathcal{F}_e(D)$ in the bilinear form $a^{swip}(u,v)$.
\end{rmk}

\subsection{Coercivity}
\label{sec:coercivity}

We choose a finite dimensional subspace $V_h(D)\subset V(D)$ and consider the problem whose solution $u_h\in V_h(D)$ satisfies
\begin{gather}
  \label{eq:dG-SWIP-D-h}
  \forall v_h\in V_h(D), \quad  a^{swip}(u_h,v_h) = b(v_h).
\end{gather}
As a closed subspace of a Hilbert space, $V_h(D)$ is a Hilbert space itself; therefore, problem~\eqref{eq:dG-SWIP-D-h} is well posed if $a^{swip}$ is coercive on $V_h(D)$.
Then $u_h$ would be a Galerkin approximation of $u$.

The bilinear form $a^{swip}$ can be proven to be coercive on $V_h(D)$ for a sufficiently high value of parameter $\sigma$ in the stabilisation form~\eqref{eq:stab-form}~\cite{DiPietro2011,Lin2016a}.
For meshes of simplices and when using polynomial spaces $V_h(D_i)$, a lower bound for $\sigma$ can be found in~\cite{Epshteyn2007}.
In this section we provide a lower bound on $\sigma$ to have the coercivity of the bilinear form $a^{swip}$ on $V_h(D)$, with an explicit expression of the coercivity constant allowing its evaluation for any finite dimensional approximation subspace of $V(D)$.

We equip the broken Sobolev space $H^1(\bigcup_{i\in I}D_i)$ with the norm $\norm{\cdot}_E$ defined by
\begin{gather}
  \norm{v}_E^2 = a(v,v)+s(v,v)+m(v,v).
\end{gather}
The application $v\mapsto (a(v,v)+s(v,v))^{1/2}$ defines a semi-norm on $H^1(\bigcup_{i\in I}D_i)$, and the addition of $m$ ensures that $\norm{\cdot}_E$ is a norm.
It is, a fortiori, a norm on $V_h(D) \subset H^1(\bigcup_{i\in I}D_i)$.

\begin{prop}[Discrete trace inequality]
  \label{prop:disc-trace-ineq}
  Let $D_i\in\mathcal{T}(D)$ and $F\in\mathcal{F}(D_i)$.
  Then
  \begin{gather}
    \exists C(V_h(D_i),F)>0, \forall v\in V_h(D_i),\quad \norm*{\nabla v_{|F}}_{L^2(F)^d} \leqslant C(V_h(D_i),F) \norm*{\nabla v}_{L^2(D_i)^d},
  \end{gather}
  where $C(V_h(D_i),F)$ depends on $V_h(D_i)$ and $F$.
\end{prop}
\begin{proof}[Proof of proposition~\ref{prop:disc-trace-ineq}]
  Let $i\in I$, $v\in V_h(D_i)$ and $F\in\mathcal{F}(D_i)$.
  From the definition of $V_h(D_i)$ above, we have
  \begin{gather}
    \norm*{\nabla v_{|F}}_{L^2(F)^d} \leqslant \norm*{\nabla v_{|F}}_{L^2(F)^d} + \norm*{\nabla v}_{L^2(D_i)^d}.
  \end{gather}
  The application   $\norm{\cdot}_{L^2(F)^d} + \norm{\cdot}_{L^2(D_i)^d}$ is a norm on the subspace $W = \{\nabla v : v\in V_h(D_i)\}$ of $L^2(D_i)$.
  Since $W$ is of finite dimension, this norm is equivalent to $\norm{\cdot}_{L^2(D_i)^d}$ on $W$, which means that there exists $C(V_h(D_i),F)>0$, independent of $v$, such that
  \begin{gather}
    \norm*{\nabla v_{|F}}_{L^2(F)^d} + \norm*{\nabla v}_{L^2(D_i)^d} \leqslant C(V_h(D_i),F)\norm*{\nabla v}_{L^2(D_i)^d}.
  \end{gather}
\end{proof}

Before stating the next result, we introduce some notations.
We denote the maximum number of faces of elements in $\mathcal{T}(D)$ by $N_{\mathcal{F}} = \max\mleft\{\#\mathcal{F}(D_i) : i\in I\mright\}$, the upper bound of face measures by $\abs{\mathcal{F}}^+=\max\{\abs{F} : F\in\mathcal{F}(D)\}$, the upper and lower bounds for the eigenvalues of the diffusion operator by $k_{max}^+ = \max \{k_i^+:i\in I\}$ and $ k_{min}^- = \min\{k_i^- : i \in I\}$, the upper bound of the average weights by $\beta_{max} = \max \mleft\{ \max\{\beta_{F}^+,\beta_{F}^-\}: F\in\mathcal{F}(D)\mright\}$, the lower bound of the weights in the stabilisation form by $\omega_{min} =\min\mleft\{\omega_F: F\in\mathcal{F}(D)\mright\}$, and the upper bound of the constant in the discrete trace inequality by $C(V_h(D)) = \max\{C(V_h(D_i),F) : i\in I, F \in \mathcal{F}(D_i)\}$.
\begin{prop}[SWIP coercivity] 
  \label{prop:coercivity}
  If 
  \begin{gather}
    \label{eq:sigma-inf}
    \sigma > \sigma_- := C(V_h(D))^2 \beta_{max}^2N_{\mathcal{F}}\abs{\mathcal{F}}^{+} \frac{k^+_{max}}{\omega_{min}} \frac{k^+_{max}}{k^-_{min} },
  \end{gather}
  then 
  \begin{gather}
    \label{eq:coercivity-swip}
    \forall v\in V_h(D),\quad
    a^{swip}(v,v) \geqslant \mleft( 1-\sqrt{\frac{\sigma_- }{ \sigma}} \mright) \norm{v}_E^2,
  \end{gather}
  \myie{} $a^{swip}$ is coercive with coercivity constant $C_{swip} = 1-\sqrt{\frac{\sigma_-}{\sigma}}$.
\end{prop}

\begin{proof}[Proof of proposition~\ref{prop:coercivity}]
  First, let us assume
  \begin{align}
    \label{eq:constantsCaCs}
    \exists C<\frac{1}{2}, \forall v\in V_h(D),\quad   c(v,v) \leqslant C (a(v,v)+s(v,v)).
  \end{align}
  Consequently, for all $v\in V_h(D)$,
  \begin{align}
    a^{swip}(v,v)  & = a(v,v)+s(v,v)-2c(v,v) + m(v,v) \\
                   & \geqslant (1-2C) (a(v,v)+s(v,v)) + m(v,v) \\
                   & \geqslant (1-2C) \norm{v}_E^2.
  \end{align}
  Therefore, it is enough that \eqref{eq:constantsCaCs} holds with $2C = \sqrt{\frac{\sigma_-}{\sigma}}$ to prove~\eqref{eq:coercivity-swip}.
  Since~\eqref{eq:sigma-inf} would then ensue from the necessary condition $C<\frac{1}{2}$, it would complete the proof.
  
  Let us consider a face $F = \partial D_i\cap \partial D_j\in \mathcal{F}(D)$.
  We let $\alpha>0$ and, applying successively Cauchy-Schwarz's and Young's inequalities, we have that
  \begin{align}
    \int_F n\cdot \{K\nabla v\}[v] & \leqslant \norm*{n\cdot\{K\nabla v\}}_{L^2(F)} \norm*{[v]}_{L^2(F)} \\
                                   &  = \frac{\alpha}{\alpha}\norm*{n\cdot\{K\nabla v\}}_{L^2(F)} \norm*{[v]}_{L^2(F)} \\
                                   & \leqslant \frac{1}{2\alpha^{2}}\norm*{n\cdot\{K\nabla v\}}_{L^2(F)}^2 + \frac{\alpha^2}{2}\norm*{[v]}_{L^2(F)}^2.
  \end{align}
  From proposition~\ref{prop:disc-trace-ineq} and the definitions of the weighted average operator and of $k_i^+$ (in equations~\eqref{eq:weighted-avg} and~\eqref{eq:def-K-bounds}), we get
  \begin{align}
    \norm*{n\cdot\{K\nabla v\}}_{L^2(F)}^2 & = \norm*{n\cdot\mleft(\beta_F^+(K\nabla v)_{|F^+} + \beta_F^-(K\nabla v)_{|F^-}\mright)}_{L^2(F)}^2 \\
                                           & \leqslant \norm*{n\cdot\beta_F^+(K\nabla v)_{|F^+}}_{L^2(F)}^2 + \norm*{n\cdot\beta_F^-(K\nabla v)_{|F^-}}_{L^2(F)}^2 \\
                                           & \leqslant \beta_{max}^2{k^+_{max}}^2 \mleft( \norm*{(\nabla v)_{|F^+}}_{L^2(F)^d}^2 + \norm*{(\nabla v)_{|F^-}}_{L^2(F)^d}^2 \mright) \\
                                           & \leqslant C(V_h(D))^2\beta_{max}^2{k^+_{max}}^2 \mleft( \norm*{(\nabla v)_{|D_i}}_{L^2(D_i)^d}^2 + \norm*{(\nabla v)_{|D_j}}_{L^2(D_j)^d}^2 \mright)
  \end{align}
  Now we let $\epsilon>0$ and choose $\alpha=C(V_h(D))\beta_{max}{k}^+_{max}\sqrt{N_{\mathcal{F}}(\epsilon k^-_{min})^{-1}}$, so that
  \begin{align}
    \int_F n\cdot \{K\nabla v\}[v] & \begin{multlined}[t]
      \leqslant \frac{1}{2\alpha^{2}}C(V_h(D))^2\beta_{max}^2{k^+_{max}}^2 \mleft( \norm*{(\nabla v)_{|D_i}}_{L^2(D_i)^d}^2 + \norm*{(\nabla v)_{|D_j}}_{L^2(D_j)^d}^2 \mright) \\
      + \frac{\alpha}{2}\norm*{[v]}_{L^2(F)}^2     
    \end{multlined} \\
    \label{eq:bound-c}
                                   & = \frac{\epsilon k^-_{min}}{2N_{\mathcal{F}}} \norm*{(\nabla v)_{|D_i}}_{L^2(D_i)^d}^2 + \frac{N_{\mathcal{F}}}{2\epsilon k^-_{min}} C(V_h(D))^2 \beta_{max}^2 {k^+_{max}}^2 \norm*{[v]}_{L^2(F)}^2
  \end{align}
  Noting that
  \begin{gather}
    k^-_{min}\norm*{\nabla v}_{L^2(D_i)^d}^2 \leqslant k^-_{i} \norm*{\nabla v}_{L^2(D_i)^d}^2 \leqslant \int_{D_i} \nabla v \cdot K \nabla v ,
  \end{gather}
  and that\footnote{Only with periodic boundary conditions (see remark~\ref{rmk:pbc-faces}), although inequality~\eqref{eq:bound-a} is still verified without them.} $\sum_{F = D_i \cap D_j \in \mathcal{F}(D)} (p_i + p_j) = N_\mathcal{F} \sum_{i\in I} p_i $, we find
  \begin{align}
    \frac{k_{min}^-}{N_\mathcal{F}} \sum_{F = D_i \cap D_j \in \mathcal{F}(D)} \mleft( \norm*{\nabla v}_{L^2(D_i)^d}^2 + \norm*{\nabla v}_{L^2(D_j)^d}^2 \mright) & = k_{min}^- \sum_{i\in I}  \norm*{\nabla v}_{L^2(D_i)^d}^2 \\
    \label{eq:bound-a}
                                                                                                                                                                  & \leqslant {a(v,v)}.
  \end{align}
  From the definition of $\omega_{min}$ and $\abs{\mathcal{F}}^+$, we also have 
  \begin{align}
    \label{eq:bound-s}
    \frac{\sigma \omega_{min}}{\abs{\mathcal{F}}^+} \sum_{F\in \mathcal{F}(D)}\norm*{[v]}_{L^2(F)}^2 \leqslant \sum_{F\in \mathcal{F}(D)} \frac{\sigma\omega_F}{\abs{F}} \int_F [v]^2 = s(v,v).
  \end{align}
  We put together~\eqref{eq:bound-a} and~\eqref{eq:bound-s} in~\eqref{eq:bound-c} and obtain
  \begin{align}
    c(v,v) = \sum_{F\in \mathcal{F}(D)} \int_F n_F\cdot \{K\nabla v\}[v] & \leqslant  \frac{C(V_h(D))^2\beta_{max}^2{k^+_{max}}^2 N_\mathcal{F}\abs{\mathcal{F}}^{+}}{2\epsilon k^-_{min}\sigma \omega_{min}}s(v,v) + \frac{ \epsilon}{2} a(v,v) \\
                                                                         & = \frac{\sigma_-}{2\epsilon\sigma}s(v,v) + \frac{ \epsilon}{2} a(v,v) \\
                                                                         & \leqslant C_\epsilon (s(v,v) + a(v,v)),
  \end{align}
  with $C_\epsilon := \max\mleft\{ \frac{\sigma_-}{2\epsilon\sigma}, \frac{\epsilon}{2} \mright\}$. Therefore, \eqref{eq:constantsCaCs} holds with $C := \inf_{\epsilon>0} C_\epsilon = \frac{1}{2} \sqrt{\frac{\sigma_-}{\sigma}}$, which concludes the proof.
\end{proof}

The result of proposition~\ref{prop:coercivity} is of major interest in choosing a suitable value for stabilisation parameter $\sigma$: too high a value degrades the performance of the algorithm that will be presented in section~\ref{sec:low-rank-appr}, due to poor conditioning of discrete operators associated with $a^{swip}$; on the other hand, $\sigma$ must be high enough for $a^{swip}$ to be coercive.
Consequently, knowledge of lower bound $\sigma_-$ enables us to set not too low a value for $\sigma$.
However, one should keep in mind that \enquote{$\sigma>\sigma_-$} is only a sufficient condition, since $\sigma_-$ is not necessarily the lowest value above which $\sigma$ ensures coercivity.
A choice of $\sigma$ lower than $\sigma_-$ may improve the performance of the aforementioned algorithm.
Alternatively, to improve conditioning while retaining coercivity, one could replace the stabilisation form $s(v,w)$ by $\sum_{F\in \mathcal{F}(D)} \int_F \sigma_F[w][v]$, where $(\sigma_F)_{F\in\mathcal{F}(D)}$ is a set of penalisation parameters defined face-wise.
Incidentally, the weights functions $\omega$ and $\beta$ added from SIP to SWIP formulations are a way of tuning the stabilisation face-wise according to conductivity.

It should be noted that, unlike typical discontinuous Galerkin settings, there are two level of discretisation here: first the mesoscopic level, at which the domain is partitioned in \enquote{cells} and where discontinuities occur; then the microscopic level, \myie{} the mesh within each cell, which relates to $V_h(D)$.
The characteristic length of the former appears in formula~\eqref{eq:sigma-inf} as $\abs{\mathcal{F}}^+$, while the latter is accounted for in $C(V_h(D))$, whose computation is discussed below.
Section~\ref{sec:numerical-results} features examples of approximation spaces with their associated trace constant's value.

\begin{rmk}[Trace constant computation]
  \label{rmk:trace-cst-comp}
  The evaluation of the lower bound $\sigma_-$ according to formula~\eqref{eq:sigma-inf} requires the evaluation of $C(V_h(D))$ which, in turn, calls for the value of $C(V_h(D_i),F)$ for all $i\in I$ and $F\in \mathcal{F}(D_i)$.
  The evaluation of $C(V_h(D_i),F)$, defined by
  \begin{gather}
    C(V_h(D_i),F)^2 = \max\mleft\{ \frac{\norm*{\nabla v_{|F}}^2_{L^2(F)^d}}{\norm*{\nabla v}_{L^2(D_i)^d}^2} : v \in V_h(D_i), \norm*{\nabla v}_{L^2(D_i)^d}>0 \mright\},
  \end{gather}
  requires computing the maximum eigenvalue of a generalised eigenvalue problem.
  Let us assume that there exists a diffeomorphism $\xi_i$ which maps $D_i$ onto a reference domain $Y$, \myie{} $\xi_i(D_i) = Y$, and that $V_h(D_i) = \{ v \circ \xi_i : x \in D_i \mapsto v(\xi_i(x)) : v \in V_h(Y)\}$, with $V_h(Y) \subset H^1(Y)$.
  If the domains $D_i$ are obtained by translations of a particular domain $ D_{i^\star} = Y$, $i^\star\in I$, then $C(V_h(D_i),F) = C(V_h(Y),F_Y)$
  , with $F_Y = \xi_i(F)$, is independent of $i$.
  If $Y=]0,1[^d$ and $\xi_i$ is an affine transformation, \myie{} $\xi_i(x) := A_i x + b_i $ for a certain invertible matrix $A_i \in \mathbb{R}^{d\times d}$ with positive determinant and a certain vector $b_i \in \mathbb{R}^{ d}$, then
  \begin{gather}
    C(V_h(D_i),F)^2 = \max\mleft\{ \frac{\abs{F}}{\abs*{D_i}}\frac{\norm*{A_i^T \nabla v_{|F_Y}}^2_{L^2(F_Y)^d}}{\norm*{A_i^T \nabla v}_{L^2(Y)^d}^2} : v \in V_h(Y), \norm*{\nabla v}_{L^2(Y)^d}>0 \mright\} ,
  \end{gather}
  so that 
  \begin{gather}
    C(V_h(D_i),F) \leqslant \frac{\varsigma_{max}(A_i)}{\varsigma_{min}(A_i)} \frac{\abs{F}^{1/2}}{\abs*{D_i}^{1/2}} C(V_h(Y),F_Y)^2,
  \end{gather}
  where $\varsigma_{max}(A_i)$ and $\varsigma_{min}(A_i)$ are respectively the maximum and minimum singular values of $A_i$.
\end{rmk}

\begin{rmk}[$K$'s eigenvalues computation]
  \label{rmk:ki-comp}
  Evaluation of the bounds $\{k_i^-,k_i^+\}_{i\in I}$ of eigenvalues of $K$, defined in~\eqref{eq:def-K-bounds}, is required for $a^{swip}$: these bounds yield the face-wise weights $\omega$, $\beta^-$ and $\beta^+$ as expressed in~\eqref{eq:def-weights}, used in bilinear forms $s$ and $c$, and are involved in formula~\eqref{eq:sigma-inf} of lower bound $\sigma_-$.
  
  If such bounds are not explicitly given, they are evaluated numerically.
  Assuming that 
  $K(x) = \sum_{i\in N_h} K(x_i) \phi_{i}(x)$ where $\{x_i : i \in  {N}_h \}$ is the set of nodes of the mesh $\mathcal{T}_h(D_i)$, and where $\{\phi_i(x) : i\in N_h\}$ forms a partition of unity, \myie{},  are non-negative functions such that $\sum_{i\in \mathcal{N}_h(D_i)} \phi_i(x) = 1$ for all $x$, then $k_i^-=\min\{\lambda_{min}(K(x_i)) : i \in  {N}_h\}$ and $k_i^+=\max\{\lambda_{max}(K(x_i)) : i\in {N}_h\}$.
  For a general $K(x)$, it can be approximated under the above form with a sufficiently fine mesh, and $k_i^-$ and $k_i^+$ are estimated from its approximation.
  
  Although, for the sake of simplicity, we consider numerical examples with a scalar-valued diffusion operator $K\in L^\infty(D)$, there is no objection to its being matrix-valued, \myie{} $K\in L^\infty(D)^{d\times d}$.
  If $K$ is diagonal, the evaluation cost of $\{k_i^-,k_i^+\}_{i\in I}$ is insignificant---a fortiori if it is scalar.
  If $K$ is not diagonal, the extreme eigenvalues of a $d$-by-$d$ matrix must be computed at every node.
  Our simulations found this latter cost, albeit not negligible, to remain small compared to the overall resolution cost.
  A parallelisation strategy would considerably reduce the cost of these evaluations.
\end{rmk}

\section{Tensor-structured method} 
\label{sec:tensor}

\subsection{Formulation over a tensor product space}
\label{sec:tens-struct-form}

We here assume that the domains $D_i$, $i\in I$, are obtained by translations of a reference domain $Y = \xi_{i}(D_i)$, with $\xi_i(x) = x + b_i$ for a certain vector $b_i \in \mathbb{R}^{ d}$.
Then there exists a bijection $\zeta$ between $I\times Y$ and $\cup_{i\in I} D_i$ (see Figure \ref{fig:domain-separation}) given by
\begin{gather}
  \zeta(i,y) = \xi_i^{-1}(y) = y-b_i, \quad (i,y) \in I \times Y.
\end{gather}
\begin{figure}[ht]
  \centering
  \includegraphics[width=.7\linewidth]{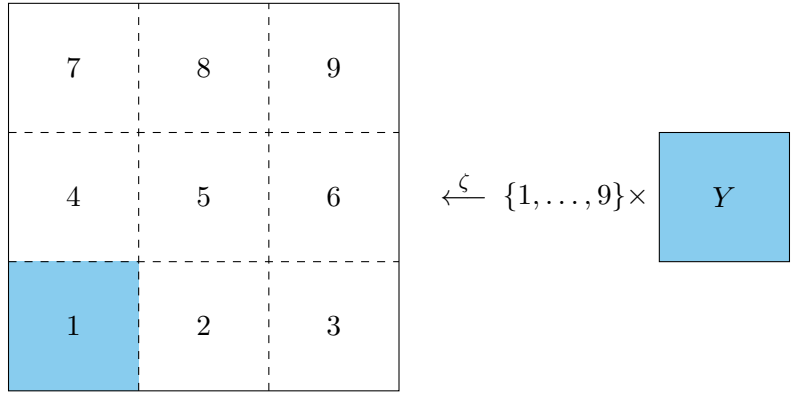}
  \caption{Bijection $\zeta$ between $I \times Y$ and $ \bigcup_{i\in I}D_i$}
  \label{fig:domain-separation}
\end{figure}

We define
\begin{gather}
  V(Y) := \mleft\{v\in H^1(Y) : (\nabla v)_{|\partial Y}\in L^2(\partial Y)^d \mright\}.
\end{gather}
Then we denote by $X = \mathbb{R}^I \otimes V(Y)$ the tensor space of functions defined on $I\times Y$ which is the linear span of elementary tensors $v^I \otimes v^Y$, with $v^I \in \mathbb{R}^I$ and $v^Y \in V(Y)$.
This tensor product space is equipped with an inner product $\langle \cdot , \cdot \rangle$ and associated norm $\norm{\cdot}$ such that
\begin{gather}
  \norm*{v^I \otimes v^Y} = \norm*{v^I}_{\mathbb{R}^I} \norm*{v^Y}_{V(Y)}.
\end{gather}
We denote by $\Upsilon$ the map which associates to a function $v : \bigcup_{i\in I} D_i \to \mathbb{R}$ the function $\Upsilon(v) = v \circ \zeta : I\times Y \to \mathbb{R}$.
This allows us to identify a function $v \in V(D)$ with a tensor $\Upsilon(v) \in X$ such that $\Upsilon(v) = \sum_{i\in I} e_i \otimes v_i^Y$, where $\{e_i\}_{i\in I}$ is the canonical orthonormal basis of $\mathbb{R}^I$, and $v_i^Y = v_{|D_i}\circ \xi_i^{-1}$.
Noting that
\begin{gather}
  \norm*{\Upsilon(v)}^2 = \sum_{i,j\in I} \langle e_i \otimes v_i^Y , e_j \otimes v_j^Y \rangle = \sum_{i\in I} \norm*{v_{i}^Y}^2_{H^1(Y)} = \sum_{i\in I} \norm*{v_{|D_i}}^2_{H^1(D_i)},
\end{gather}
we have that $\Upsilon$ defines a linear isometry between $V(D)$ and $X$, with $V(D)$ equipped with the norm $\norm{\cdot}_{V(D)}$ defined by $\norm{v}_{V(D)}^2 = \sum_{i\in I} \norm*{v_{|D_i}}^2_{H^1(D_i)}$, which is equivalent to the energy norm $\norm{\cdot}_E$.
Then $a^{swip}$ can be identified with a bilinear form on $X\times X$ of the form
\begin{gather}
  \label{aswip-repres}
  a^{swip} = \sum_{k=1}^{r_{swip}} a^I_k \otimes a^Y_k
\end{gather}
for some bilinear forms $a^I_k : \mathbb{R}^I\times\mathbb{R}^I \to \mathbb{R}$ and $a^Y_k : V(Y)\times V(Y) \to \mathbb{R}$ to be determined.
The bilinear forms $a^I_k$ are here identified with matrices in $\mathbb{R}^{I\times I}.$ Similarly, $b$ can be identified with a linear form on $X$ of the form
\begin{gather}
  \label{b-repres}
  b  = \sum_{k=1}^{r_b} b^I_k \otimes b^Y_k
\end{gather}
for some linear forms $b^I_k : \mathbb{R}^I \to \mathbb{R}$ and $b^Y_k : V(Y) \to \mathbb{R}$ to be determined.
The linear forms $b^I_k$ are identified with vectors in $\mathbb{R}^I$.
Subsequently, as with~\eqref{eq:dG-SWIP-D} the tensor representation $u\in X$ of the solution to problem~\eqref{eq:diffusion-strong} verifies $a^{swip}(u,v) = b(v)$ for all $v\in X$.

We choose a finite dimensional subspace $V_h(Y)\subset V(Y)$, as we did with $V_h(D)$ in section~\ref{sec:coercivity}.
This defines another finite dimensional subspace $X_h:= \mathbb{R}^I\otimes V_h(Y)\subset X$.
Approximation subspaces $V_h(D)$ and $X_h(D)$ are linearly isometric, and problem~\eqref{eq:dG-SWIP-D-h} is then equivalent to finding a tensor $u_h \in X_h$ such that
\begin{gather}
  \label{eq:diff-SWIP-IY}
  \forall v_h \in X_h, \quad a^{swip}(u_h,v_h) = b(v_h).
\end{gather}

For a comprehensive introduction to tensor numerical calculus and problems formulated over tensor spaces, we refer the reader to the monograph~\cite{Hackbusch2012a}.

\paragraph*{Representation of  the linear form $b$ on $X$.}
To obtain a representation of the linear form $b$ in the form \eqref{b-repres}, it is sufficient to consider the restriction of $b$ to elementary tensors.
Let us assume that the source term $f$ is such that
\begin{align}
  \label{lowrankf}
  \Upsilon({f}) = \sum_{k=1}^{r_f} f^I_k \otimes f^Y_k \;\in \mathbb{R}^I\otimes L^2(Y);
\end{align}
see remark~\ref{rmk:K-f-tensor} on this representation.
For $v \in V(D)$ such that  $\Upsilon(v) = v^I \otimes v^Y$, with $v^I\in \mathbb{R}^I$ and $v^Y \in V(Y)$, we then have 
\begin{align}
  b(v) = \sum_{i\in I}\int_{D_i} fv = \sum_{k=1}^{r_f} \sum_{i\in I} v^I(i)f^I_k(i) \int_Yf_k^Yv^Y,
\end{align}
which yields a representation of the form \eqref{b-repres} with $r_b=r_f$ and linear forms $b^I_k(v^I) = \sum_{i\in I} v^I(i)f^I_k(i)$ and $b_k^Y(v^Y) = \int_Yf_k^Yv^Y$.
Note that $b^I_k$ can be identified with the vector $f^I_k$.

\paragraph*{Representation of $a^{swip}$ on $X\times X$.}
To obtain a representation of the bilinear form $a^{swip}$ in the form \eqref{aswip-repres}, it is sufficient to consider the restriction of $a^{swip}$ to elementary tensors.
We first consider the representation of the diffusion form $a$.
Let us assume that the conductivity field $K$ is such that
\begin{align}
  \label{lowrankK}
  \Upsilon(K) = \sum_{n=1}^{r_K} K^I_n \otimes K^Y_n \;\in \mathbb{R}^I\otimes L^\infty(Y)
\end{align}
(see remark~\ref{rmk:K-f-tensor}).
Then, for any $v,w$ in $V(D)$ such that $\Upsilon(v) = v^I\otimes v^Y$ and $\Upsilon(w) =w^I\otimes w^Y$ are elementary tensors in $X$, we have
\begin{align}
  a(v,w) = \sum_{i\in I}\int_{D_i} K\nabla v\cdot\nabla w = \sum_{n=1}^{r_K}\sum_{i\in I} K_n^I(i)v^I(i)w^I(i) \int_Y K_n^Y\nabla v^Y\cdot\nabla w^Y,
\end{align}
which yields 
\begin{align}
  a  =	 \sum_{n=1}^{r_K}\diag(K^I_n)\otimes N[K^Y_n], 
\end{align}    
with $\diag(K^I_n)$ the diagonal matrix in $\mathbb{R}^{I\times I}$ with diagonal $K^I_n$, and $N[\psi]$ the bilinear form defined for $\psi\in L^\infty(Y)$ by
\begin{align}
  N[\psi](v^Y,w^Y) = \int_Y \psi\nabla v^Y \cdot\nabla w^Y.
\end{align}
In a similar way, we obtain 
\begin{multline}
  c  = \sum_{n=1}^{r_K}\sum_{q=1}^d l\mleft(\chi^q[K^I_n]^T\mright)\otimes N_0^q[K^Y_n] + l\mleft(\chi^q[K^I_n]\mright)\otimes N_0^{-q}[K^Y_n] \\
  - \chi^q[K^I_n]^T\otimes N_1^q[K^Y_n] - \chi^q[K^I_n]\otimes N_1^{-q}[K^Y_n], 
\end{multline}
and 
\begin{multline}
  s  = \sigma\sum_{q=1}^d l\mleft(\chi^q\mleft[\frac{\omega_K}{\abs*{\partial Y_q}}\mright]^T\mright)\otimes M_0^q + l\mleft(\chi^q\mleft[\frac{\omega_K}{\abs*{\partial Y_q}}\mright]\mright)\otimes M_0^{-q} \\
  - \chi^q\mleft[\frac{\omega_K}{\abs*{\partial Y_q}}\mright]^T\otimes M_1^q 
  - \chi^q\mleft[\frac{\omega_K}{\abs*{\partial Y_q}}\mright]\otimes (M_1^q)^T ,
\end{multline}
where the bilinear forms $ M_0^q$, $M_1^q$, $N_0^q[\psi]$ and $N_1^q[\psi]$ are respectively defined, for $q\in\{-d,\dots,d\}\setminus\{0\}$, by 
\begin{align}
  M_0^q(v^Y,w^Y) & = \int_{\partial Y_q} v^Y w^Y, &    M_1^q(v^Y,w^Y) & = \int_{\partial Y_q} v^Y(w^Y\circ\tau_q), 
  \\
  N_0^q[\psi](v^Y,w^Y) & = \int_{\partial Y_q} \psi\frac{e_q}{2}\cdot(\nabla v^Y)w^Y, & N_1^q[\psi](v^Y,w^Y) & = \int_{\partial Y_q} \psi\frac{e_q}{2}\cdot\nabla v^Y(w^Y\circ\tau_q),
\end{align}
with $(e_q)_{q\in\{1,\ldots,d\}}$ the canonical basis of $\mathbb{R}^d$ and $e_{-q}:=-e_q$, $\partial Y_q$ the face of $Y$ whose outward normal is $e_q$ and $\tau_q$ the translation that maps $\partial Y_q$ onto $\partial Y_{-q}$, where the matrix $\chi^q[\psi]$ is defined by
\begin{gather}
  \mleft(\chi^q[\psi]\mright)_{ij} =
  \begin{cases}
    \psi(i,j) & \text{if } \xi_i(\partial D_i \cap \partial D_j)=\partial Y_q \\
    0 & \text{else}
  \end{cases},	
\end{gather}
and where for a matrix $A \in \mathbb{R}^{I\times I}$, $l(A)$ is the diagonal matrix such that $ l\mleft(A\mright)_{ij} = \delta_{ij}\sum_{k\in I} (A)_{ik}$.
Finally, we have
\begin{align}
  m \equiv \mathds{1}^I\otimes M ,
\end{align}
with $\mathds{1}^I$ the identity matrix in $\mathbb{R}^{I\times I}$ and 
\begin{align}
  M(v^Y,w^Y)  = \int_Y v^Y w^Y.
\end{align}

\begin{rmk}[Tensor representations of $K$ and $f$]
  \label{rmk:K-f-tensor}
  The formulation of problem~\eqref{eq:diff-SWIP-IY} over tensor product space $X_h$ requires knowledge of tensor representations $\Upsilon(K)\in \mathbb{R}^I\otimes L^\infty(Y)$ and  $\Upsilon(f)\in \mathbb{R}^I\otimes L^2(Y)$, yet they are generally known as elements of $L^\infty(D)$ and $L^2(D)$, respectively.
  We showed that there is a straightforward identification of $K$   with $\sum_{k=1}^{\# I}e_k\otimes (K_{|D_k}\circ\xi_i^{-1})$, and likewise for $f$.
  This representation of $K$ involved the sum of  $\#I$ elementary tensor products, which would lead to representations of $a^{swip}$ and $b$ with an even greater number of terms, hence high storage and computational complexities.
  This would degrade the performance of the algorithm that is to be introduced in section~\ref{sec:low-rank-appr}.
  Therefore, it is desirable to look for tensor representations  in the form~\eqref{lowrankK} and~\eqref{lowrankf} with a small number of terms, \myie{}, low rank 
  $\rank(K)$ and $\rank(f)$, respectively.

  Apart from rare simple cases (such as the examples in section~\ref{sec:numerical-results}), $K$ and $f$ have full rank, so that low-rank approximations have to be introduced. 
  Such approximations can be sought by using truncated singular value decomposition or empirical interpolation method~\cite{Maday2009a}.
  Thus the ranks of $a^{swip}$ and $b$ are curbed, which improves computational efficiency.
  From the quasi-periodicity assumption, $K$ is expected to have a low rank or, at least, to admit an accurate low-rank approximation.
\end{rmk}

\subsection{Low-rank approximation}
\label{sec:low-rank-appr}

Tensor-based approaches have already been successfully used to reduce multiscale complexity, \myeg{} by exploiting sparsity in~\cite{Hoang2005a}. 
The novelty of the method presented here lies in the tensor representation designed specifically to exploit quasi-periodicity via low-rank approximation techniques.

In order to get some insight into the relation between quasi-periodicity and low-rankness,
we first note that a periodic function $v: D\to\mathbb{R}$ is such that for all $i\in I$, $v_{|D_i} = v^Y \circ \xi_i$, with $v^Y: Y\to\mathbb{R}$.
Such a function is identified with the rank-1 tensor $\Upsilon(v) = {1}_I \otimes v^Y,$ where $1_I(i)=1$ for all $i\in I$.
Let us now consider a function $v$ which coincides with a periodic function except on a subset of cells indexed by $\Lambda \subset I$.
The function $v$ is such that $v_{|D_i} = v_0^Y \circ \xi_i$ for all $i \in I\setminus \Lambda$, and $v_{|D_i} = v_i^Y \circ \xi_i$ for $i \in \Lambda$, where $v_0^Y$ and $v_i^Y$, $i\in \Lambda$, are scalar functions defined over $Y$. 
Then, $v$ can be identified with a tensor
\begin{gather}
  \Upsilon(v) = 1_{I\setminus \Lambda} \otimes v_0^Y + \sum_{i\in \Lambda} 1_{\{i\}} \otimes v_i^Y,
\end{gather}
where, for $A\subset I$, $1_A$ is such that $1_A(i) = 1$ if $i\in A$ and $0$ if $i\notin A$; the rank of this tensor is bounded by $1+\#\Lambda$.
A function which coincides with a periodic function except on a small number of cells will therefore admit a representation as a tensor with low-rank.
Figure~\ref{fig:rank2} illustrates this case for $\#\Lambda=1$.
Also, note that even if $\#\Lambda$ is large but many of the functions $v_i^Y$ are the same, then the rank may be low.
More precisely, $\rank(v) \leqslant 1+ \vdim(\vspan\{v_i\}_{i\in \Lambda})$.
We expect that the solution of~\eqref{eq:diff-SWIP-IY}, for a quasi-periodic medium and for some right-hand sides, will admit an accurate approximation with such a function.

\begin{figure}
  \centering
  \includegraphics[width=0.4\linewidth,height=0.4\linewidth]{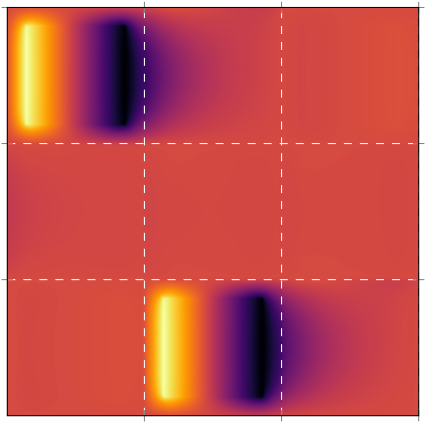}
  \caption{Example of rank-2 function}
  \label{fig:rank2}
\end{figure}

In order to build a low-rank approximation of the solution of~\eqref{eq:diff-SWIP-IY}, various algorithms are available in the literature; the reader may consult surveys~\cites{Khoromskij2012a}{Grasedyck2013a} for a presentation of existing methods.
We here rely on an adaptive algorithm detailed in~\cite{Nouy2017b}, which we will outline below. 
Let $J$ be the convex functional given by
\begin{gather}
  J(v) = \frac{1}{2}a^{swip}(v,v) - b(v,v),
\end{gather}
whose unique minimiser over $X_h$ is the solution $u$ of~\eqref{eq:diff-SWIP-IY}.
This algorithm constructs a sequence of approximations $(u_n)_{n\geqslant 1}$ with increasing rank, starting with $u_0=0$.
At each step $n\geqslant 1$, it proceeds as follows.
A rank-one correction $u^I_n\otimes u^Y_n$ of $u_{n-1}$ is first computed by solving the optimisation problem
\begin{gather}
  \label{eq:greedy-correction}
  \min_{v^I \in \mathbb{R}^I , v^Y \in V_h(Y)} J(u_{n-1}+v^I\otimes v^Y).
\end{gather}
In practice, we perform a few iterations of an alternating minimisation algorithm which consists in minimising alternatively over $u^I$ and $u^Y$.
This first step yields an approximation $ u_{n}$ of the form $u_n = \sum_{k=1}^{n} u^I_k\otimes u^Y_k.$ Then we compute the Galerkin projection of the solution in $\mathbb{R}^I \otimes U_n$, with $U_n = \vspan\{u^Y_1,\ldots, u^Y_n\},$ which is solution of
\begin{gather}
  \label{eq:greedy-updateY}
  \min_{v \in \mathbb{R}^I \otimes U_n} J(v).
\end{gather}
This is equivalent to updating the functions $u^I_k$ in the representation of $u_n$ by minimising $J(u_n)$ over the functions $u^I_k$ with fixed functions $u^Y_k$.
Finally, we compute the Galerkin projection of the solution in $S_n \otimes V_h(Y)$, with $S_n = \vspan\{u^I_1,\ldots, u^I_n\}$, which is solution of
\begin{gather}
  \label{eq:greedy-updateI}
  \min_{v \in S_n \otimes V_h(Y)} J(v).
\end{gather}
This is equivalent to updating the functions $u^Y_k$ in the representation of $u_n$ by minimising $J(u_n)$ over the functions $u^Y_k$ for fixed functions $u^I_k$.
Finally, we stop the algorithm when the residual error criterion
\begin{gather}
  \label{eq:residual-error}
  \frac{\norm*{a^{swip}(u_n,\cdot) - b}_{X_h'}}{\norm*{b}_{X_h'}} \leqslant \text{tolerance}
\end{gather}
is verified.
This method is a particular case of one of the class of algorithms whose convergence analysis can be found in~\cite{Falco2012}.

We may give some insight into the complexity reduction through problems sizes.
A direct resolution of~\eqref{eq:dG-SWIP-D} requires the solution of a linear system of size $\#I\times \vdim(V_h(Y))$.
One step of the proposed algorithm requires the alternate solution of problems of size $\# I$ and $\dim(V_h(Y))$ in the rank-one correction step, the solution of one problem of size $n \times \#I$ and finally, the solution of one problem of size $n\times \dim(V_h(Y))$.
The cost of one iteration therefore increases with $r$ but for moderate ranks $n$, it remains small compared to a direct solution method.
Note also that compared to a direct solution method, the tensor-structured approach may allow a significant reduction in the storage of the operator.

\section{Numerical results}
\label{sec:numerical-results}

The proposed multiscale low-rank approximation method, here denoted MsLRM, has been tested on two-dimensional problems with quasi-periodic diffusion operator $K$ of the form
\begin{gather}
  \label{eq:test-K-gen}
  \Upsilon(K) = B \otimes K^Y_1 + (1-B) \otimes K^Y_2,
\end{gather}
where the $\{B(i) : i \in I\}$ are independent and identically distributed Bernoulli random variables with values in $\{0,1\}$.
This means that $K$ is a random function whose restriction to any cell $D_i$ is $K^Y_1$ if $B(i)=1$, and $K^Y_2$ if $B(i)=0$. 
The conductivity field $K$ can be interpreted as a random perturbation of an ideal periodic medium, where a cell $D_i$ displays the material property of the reference periodic medium if $B(i)=1$, and a \enquote{perturbed} property if $B(i)=0$.
This arbitrary interpretation means that $K^Y_1$ represents the conductivity of a \emph{sound} cell and $K^Y_2$ the conductivity of a \emph{faulty} one.
Since the  $\{B(i) : i \in I\}$ are identically distributed, the defect probability is the same for every cells and we note it $p:=\mathbb{P}(B(i)=0)$.

The source term chosen is the same for all experiments and was inspired by corrector problems in stochastic homogenisation~\cite{Anantharaman2011a}.
We define it over $D$ as
\begin{gather}
  \label{eq:src-corr1}
  f = \nabla\cdot(Ke_1),
\end{gather}
where the choice of direction $e_1$ is arbitrary.
The boundary conditions remain periodic.

\begin{table}
  \centering
  \caption{Default parameters}
  \begin{tabular}{cccc}\toprule
    $\vdim(V_h(Y))$ & Tolerance & $p$  & $\sup(K)/\inf(K)$ \\\midrule
    \num{441}           & \num{e-3}          & \num{.1} & \num{100}             \\\bottomrule
  \end{tabular}
  \label{tab:param-ref}
\end{table}

We choose an approximation space $V_h(Y)$ of continuous, piecewise affine functions\footnote{Those are piecewise Lagrange polynomials of degree at most 1.} based on a mesh of isoparametric quadrangle elements.
This mesh is a regular grid of $20\times20$ elements; figure~\ref{fig:mesh-cell} shows an isotropic example for $Y:=]0,1[^d$.
The associated trace constant $C(V_h(Y))\approx 7$ (unless specified otherwise), computed accordingly to remark~\ref{rmk:trace-cst-comp}.
For comparison, we use as a reference method a standard continuous Galerkin finite element method with an approximation space $V_{h,per}^c(D) = V_h(D)\cap C^0(D) \cap H^1_{per}(D)$ (continuous and periodic functions in $V_h(D)$).

Unless specified otherwise, the parameters default values given in table~\ref{tab:param-ref} apply.

\begin{rmk}[Approximation spaces' dimensions compared]
  \label{rmk:dim-comparison}
  Where elements of $V_{h,per}^c(D)$ are concerned, each cell is meshed as $Y$ is (see an example on figure~\ref{fig:mesh-fem}) and therefore  $\vdim(V_{h,per}^c(D))$ is of the same order as $\vdim(V_h(D))$.
  More precisely, $\vdim(V_{h,per}^c(D)) < \vdim(V_h(D))$ because of the continuity constraints at cell interfaces---including half of external faces, due to periodic boundary conditions.
  Those cell interfaces are outlined on the example of figure~\ref{fig:mesh-fem}; each node located along those lines would have one more degree of freedom in $V_h(D)$ than in $V_{h,per}^c(D)$.
  For example, a square domain of \num{1024} cells with $\vdim(V_h(Y))=\num{441}$ yields $\vdim(V_{h,per}^c(D))=\num{410881}$, whereas $\vdim(V_h(D))=\num{451584}$.
\end{rmk}

\begin{figure}
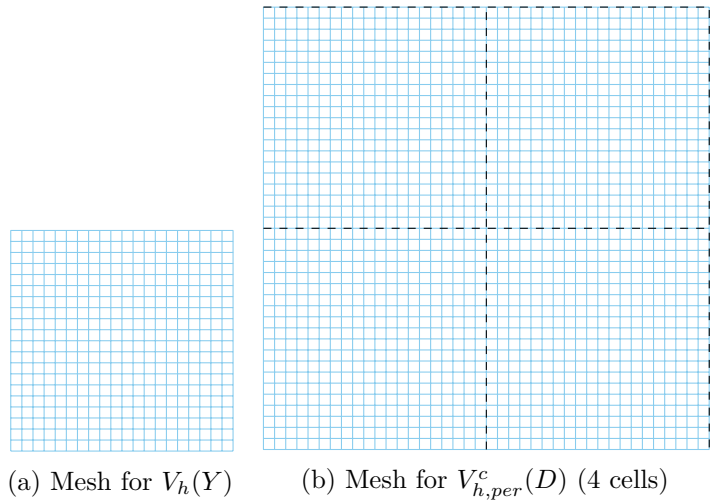

  \centering
  \begin{subfigure}[b]{.2\linewidth}
    \includegraphics[width=\linewidth]{mesh_cell.tikz}
    \caption{Mesh for $V_h(Y)$}
    \label{fig:mesh-cell}
  \end{subfigure}
  ~
  \begin{subfigure}[b]{.4\linewidth}
    \includegraphics[width=\linewidth]{mesh_multi-cells.tikz}
    \caption{Mesh for $V_{h,per}^c(D)$ (\num{4} cells)}
    \label{fig:mesh-fem}
  \end{subfigure}
  \caption{Meshes comparison between MsLRM and FEM}
  \label{fig:meshes-comparison}
\end{figure}

All computations were run on the same workstation, \myviz{} a Dell\texttrademark{} Optiplex\texttrademark{} 7010 with:
\begin{itemize}
\item \SI{8}{\gibi\byte}\footnote{\num{5.5} to \num{6.5} of which were usually available for the simulations.} ($2\times4$) RAM DDR3 \SI{1600}{\mega\hertz};
\item Intel\textregistered{} Core\texttrademark{} i7-3770 CPU: 4 cores at \SI{3.40}{\giga\hertz} with 2 threads each.
\end{itemize}

\subsection{Various conductivity patterns}
\label{sec:cond-patterns}

Here we compare the computational time between FEM and MsLRM on three test cases.
These differ by their diffusion operator $K$, reference cell $Y$ and connectivity between cells.
For each case, the comparison spans three values of $\#I$ (\myviz{} \numlist{25;100;225}) to give an small insight into computational cost sensitivity to an increase in number of cells.

\subsubsection*{Missing fibres}

This test case was directly inspired by composite materials with unidirectional fibre reinforcements.
The mesoscopic mesh is also unidirectional since every cell spans the entire width of the domain, with $Y:=]0,1[\times ]0,5[$ and $D:=[0,\#I]\times[0,5]$.
Fibre and matrix both have uniform conductivities and the faulty cells have no fibre, thus $K$ is expressed as~\eqref{eq:test-K-gen} with $K_1^Y = 1 + 99\chi$ and $K_2^Y = 1$, where $\chi\in C^0(Y,[0,1])$ is the continuous indicator function of the fibre, \myie{} $[0.25,0.75]\times]0,1[$.
An example of such conductivity with five cells, of which the middle one is faulty, is displayed on figure~\ref{fig:k-missing-fibres}.

$Y$ is meshed with the same number of elements as on figure~\ref{fig:mesh-cell}, resulting in an anisotropic mesh.
We thus keep $\vdim(V_h(Y))$ at the value in table~\ref{tab:param-ref} but, due to these particular cell size and mesh, the trace constant here is $C(V_h(Y))\approx 4.47$.

The results are shown in table~\ref{tab:missing-fibres}.
Only a rank 3 is required to reach the desired precision and therefore we hardly see any effect of the increase in domain size on computational time.
These results are mainly due to the unidimensionality of the mesoscopic mesh.
A faulty cell essentially affects the solution in the two neighbouring cells, as is visible on figures~\ref{fig:approx-missing-fibres} and~\ref{fig:sol-missing-fibres} which display the reference FEM solution and its MsLRM approximation for the conductivity from figure~\ref{fig:k-missing-fibres}.

For the FEM resolution, we observe a more significant increase in computational time as $\#I$ increases.

\begin{figure}
  \centering
  \begin{subfigure}{.4\linewidth}
    \centering
    \includegraphics[width=\textwidth]{missing-fibres-cond.tikz}
    \caption{Conductivity}
    \label{fig:k-missing-fibres}
  \end{subfigure}
  ~  
  \begin{subfigure}{.4\linewidth}
    \centering{}
    \includegraphics[width=\textwidth,height=\textwidth]{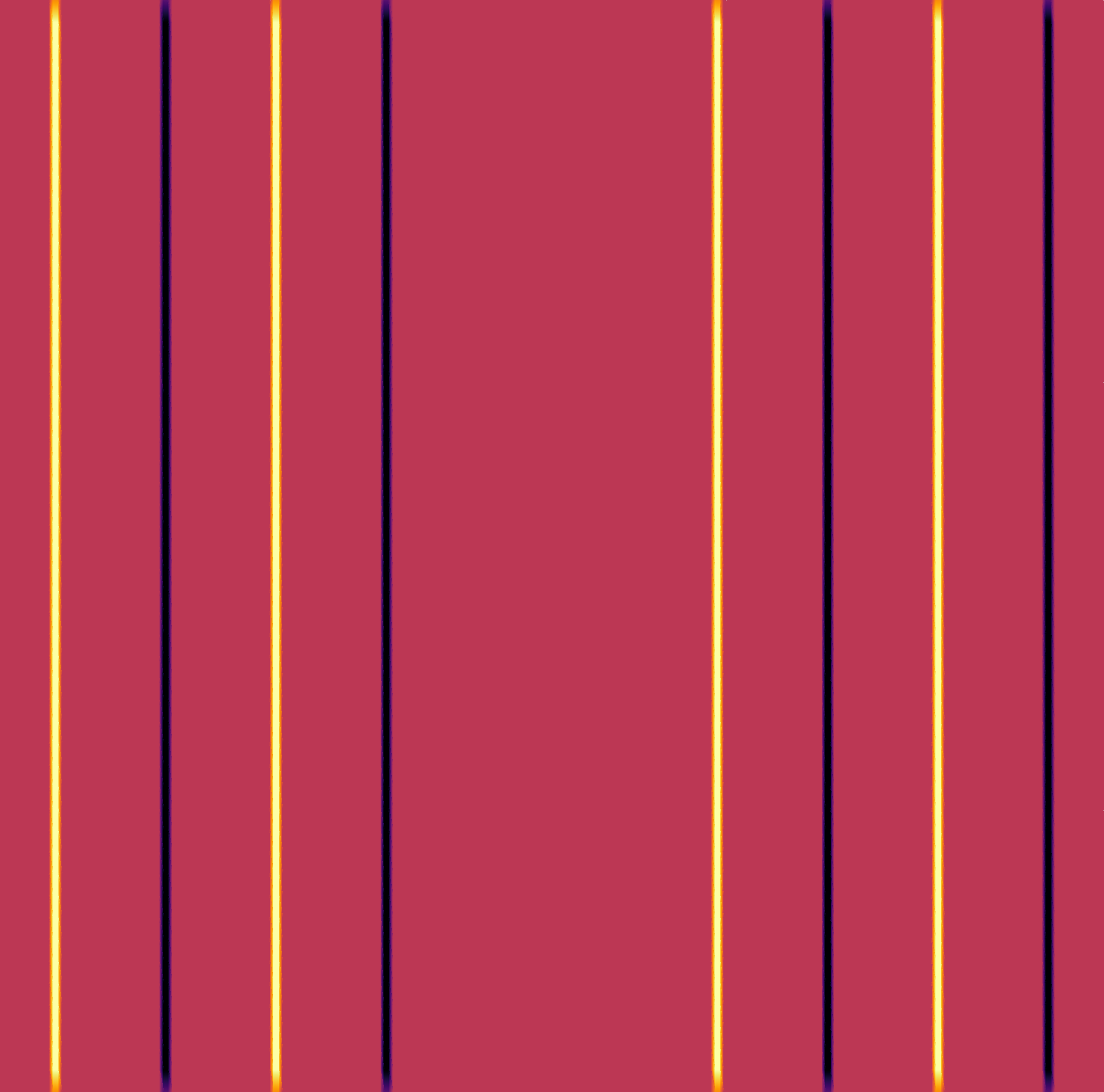}
    \caption{Source term}
  \end{subfigure}

  \begin{subfigure}{.4\linewidth}
    \centering{}
    \includegraphics[width=\textwidth,height=\textwidth]{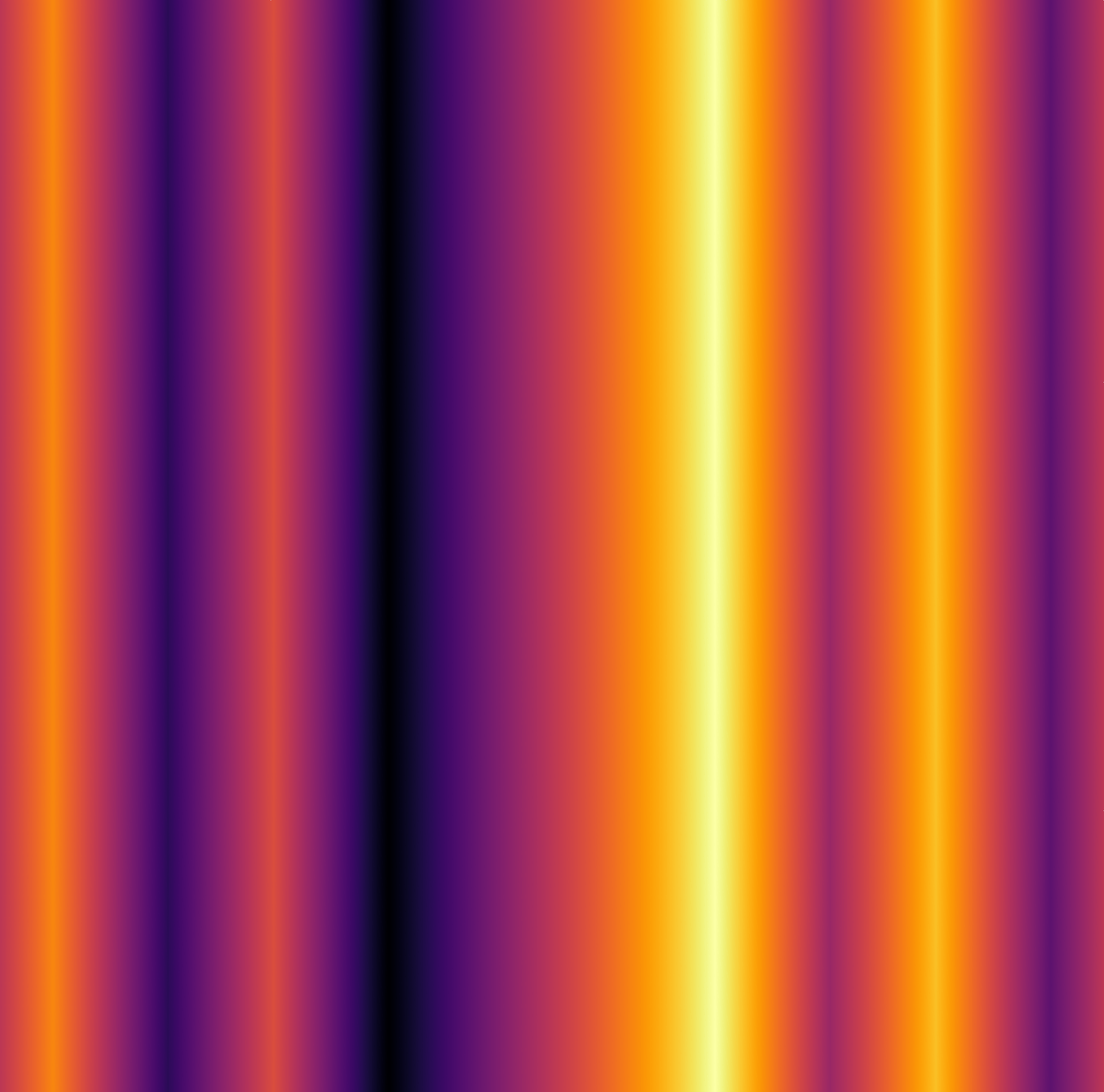}
    \caption{Rank-3 approximation}
    \label{fig:approx-missing-fibres}
  \end{subfigure}
  ~
  \begin{subfigure}{.4\linewidth}
    \centering{}
    \includegraphics[width=\textwidth,height=\textwidth]{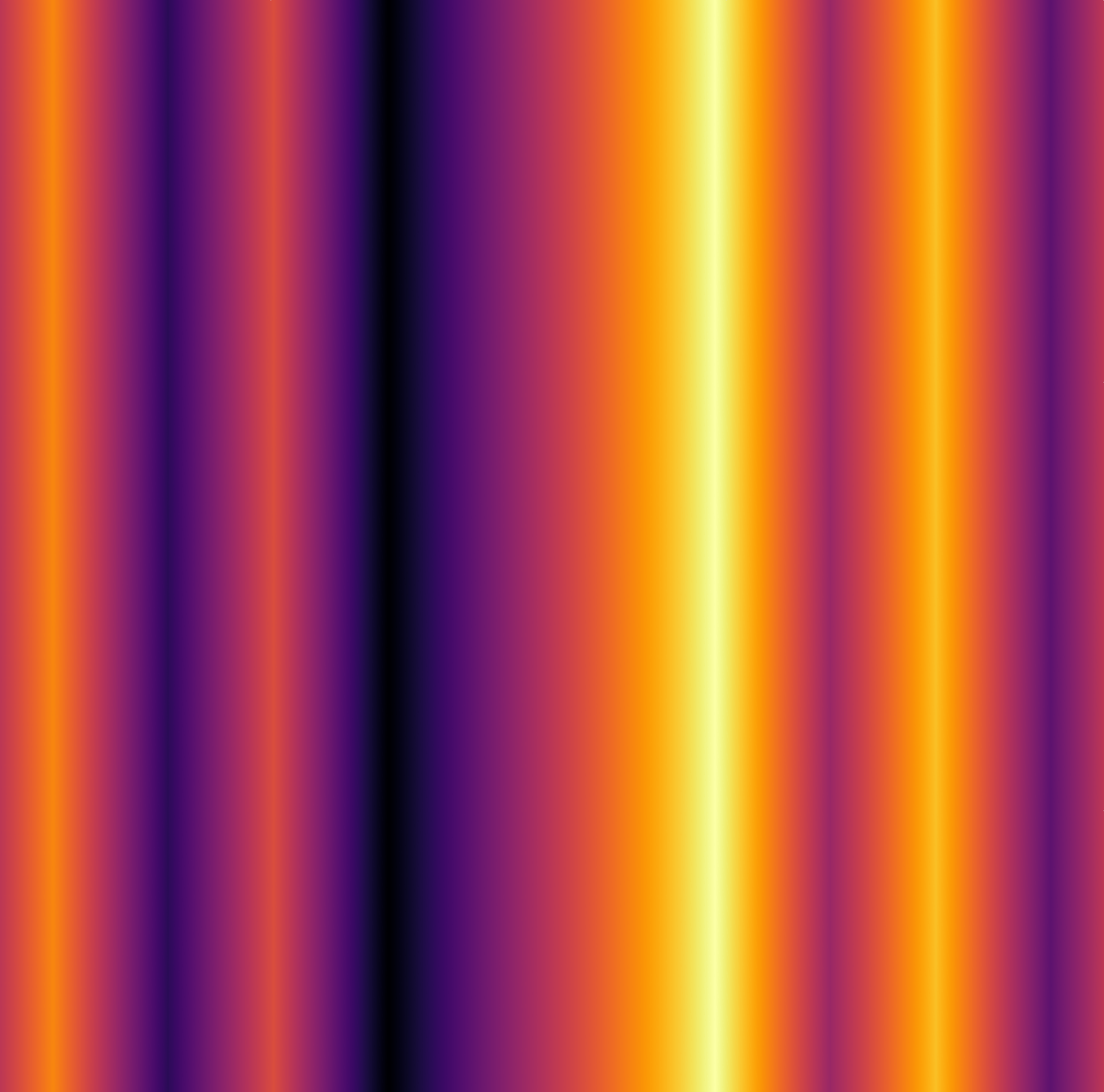}
    \caption{Reference solution}
    \label{fig:sol-missing-fibres}
  \end{subfigure}
  \caption{Missing fibres test case: example with five cells}
  \label{fig:missing-fibres}
\end{figure}

\begin{table}
  \centering
  \caption{Missing fibres test case}
  \begin{tabular}{*{4}{S}}\toprule
    {\multirow{2}*{$\#I$}} & {FEM}      & \multicolumn{2}{c}{MsLRM} \\\cmidrule(l){3-4}
                           & {time (s)} & {time (s)} & {rank}       \\\midrule
    25                     & 1.3        & 0.57       & 3            \\
    100                    & 9          & 0.60       & 3            \\
    225                    & 35         & 0.56       & 3            \\\bottomrule
  \end{tabular}
  \label{tab:missing-fibres}
\end{table}

\subsubsection*{Undulating fibres}

This test was inspired by woven composite materials.
Unlike the previous test, there are fibres in two orthogonal directions and the mesoscopic mesh is bidimensional with reference cell $Y:=]0,1[^d$.
Faulty cells show an undulation in a fibre, as illustrated on figure~\ref{fig:k-undulating-crosses}.
Consequently, $K$ is expressed as in~\eqref{eq:test-K-gen} with $K_1^Y = 1 + 99\chi_1$ and $K_2^Y = 1 + 99\chi_2$, where $\chi_1,\chi_2\in C^0(Y,[0,1])$ are indicator functions of the straight cross and cross with bent fibre, respectively; the crosses' arms have a width of $1-2^{-1/2}$ so that they occupy half the surface.

The results in table~\ref{tab:undulating-crosses} show that, compared to the first test case, a higher rank of approximation is necessary to achieve the same precision.
This is mainly due to the bidimensionality of the mesoscopic mesh: each cell has eight neighbours, whereas it had only two in the first test case.
The impact of a defect requires more functions in $V_h(Y)$ to be represented.
Reference solution and its approximation are displayed on figures~\ref{fig:sol-undulating-fibres} and~\ref{fig:approx-undulating-fibres}.

Consequently, the computational time is more affected by an increase in the number of cells.
This increase in computational time remains, however, considerably smaller  than that of the reference solution method.

\begin{figure}
  \centering
  \begin{subfigure}{.4\linewidth}
    \centering
    \includegraphics[width=\textwidth]{undulating-crosses-cond.tikz}
    \caption{Conductivity}
    \label{fig:k-undulating-crosses}
  \end{subfigure}
  ~ 
  \begin{subfigure}{.4\linewidth}
    \centering{}
    \includegraphics[width=\textwidth,height=\textwidth]{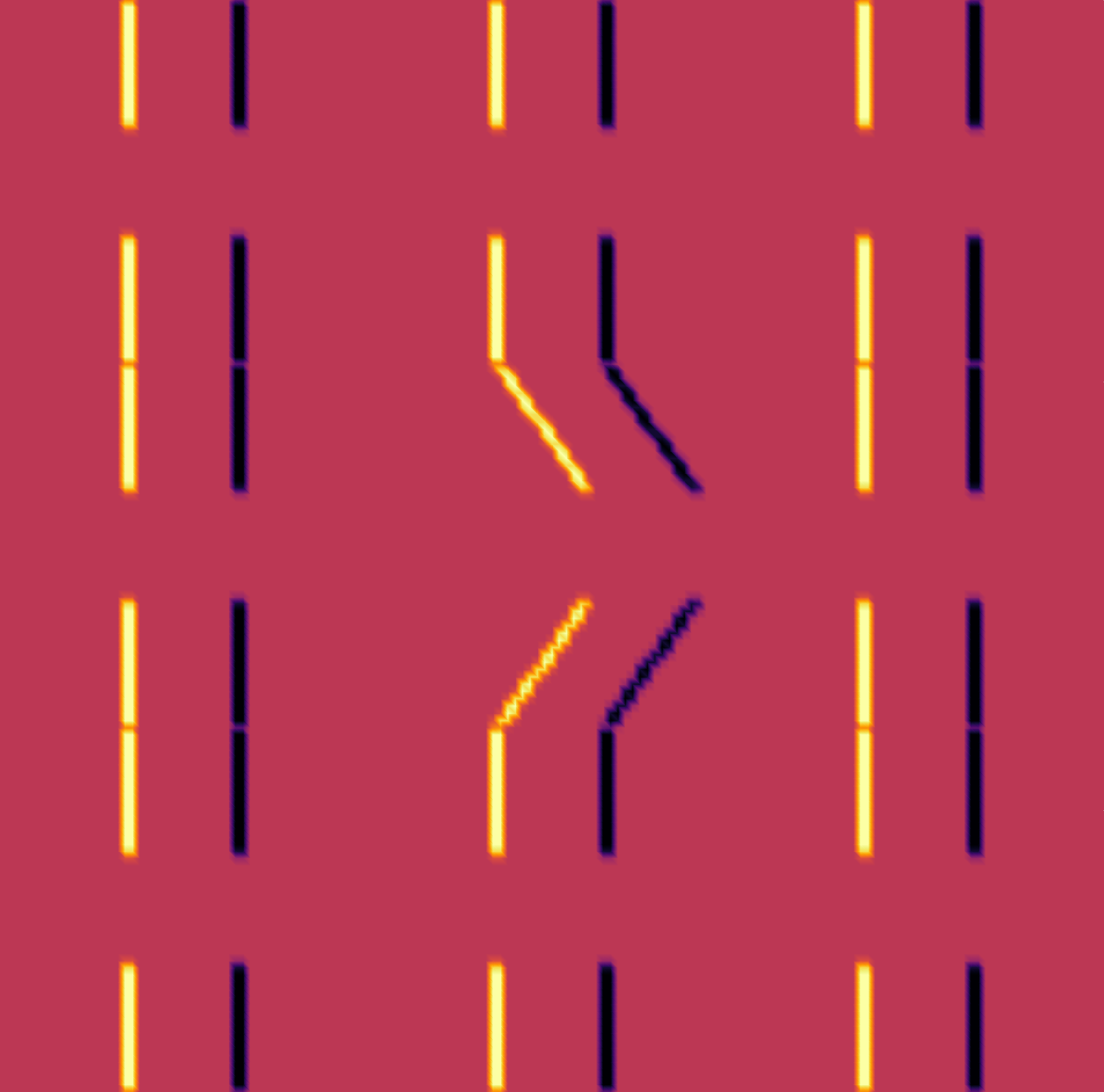}
    \caption{Source term}
  \end{subfigure}

  \begin{subfigure}{.4\linewidth}
    \centering{}
    \includegraphics[width=\textwidth,height=\textwidth]{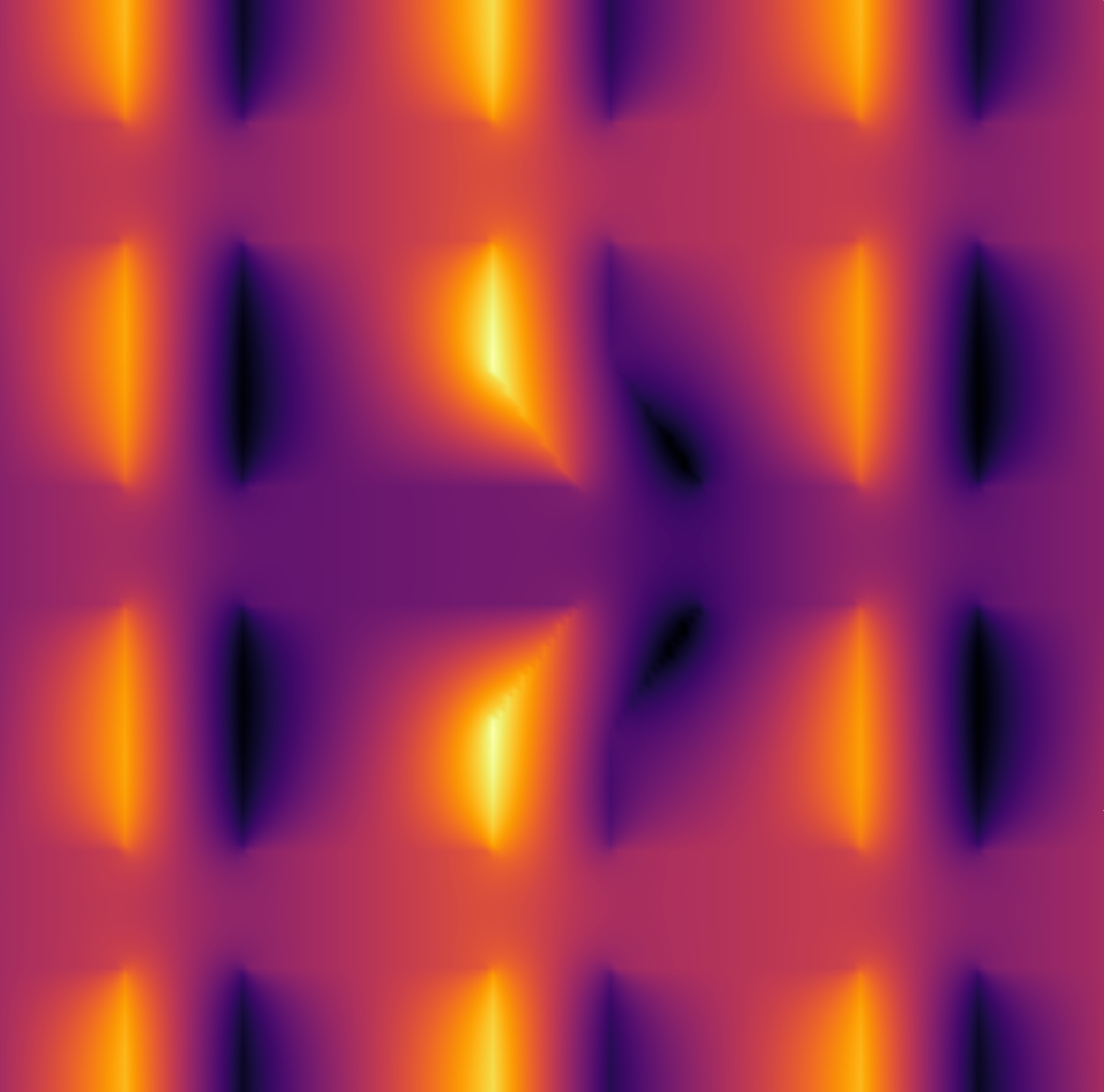}
    \caption{MsLRM approximation}
    \label{fig:sol-undulating-fibres}
  \end{subfigure}
  ~
  \begin{subfigure}{.4\linewidth}
    \centering{}
    \includegraphics[width=\textwidth,height=\textwidth]{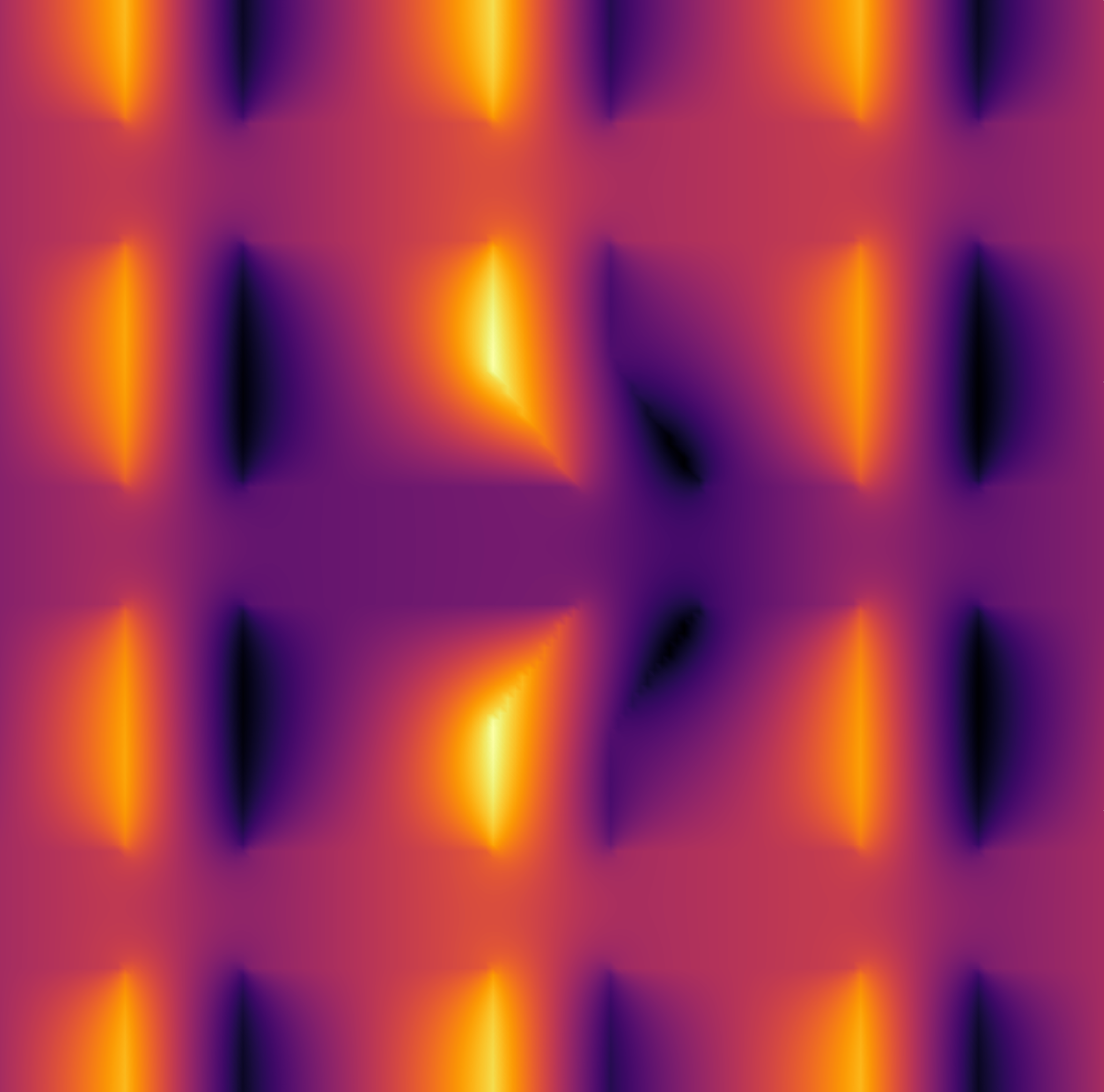}
    \caption{Reference solution}
    \label{fig:approx-undulating-fibres}
  \end{subfigure}
  \caption{Undulating fibres test case: example with nine cells}
\end{figure}

\begin{table}
  \centering
  \caption{Undulating fibres test case}
  \begin{tabular}{*{4}{S}}\toprule
    {\multirow{2}*{$\#I$}} & {FEM}      & \multicolumn{2}{c}{MsLRM} \\\cmidrule(l){3-4}
                           & {time (s)} & {time (s)} & {rank}       \\\midrule
    25                     & 1.3        & 2.16       & 14           \\
    100                    & 9          & 5.1        & 20           \\
    225                    & 35         & 6.4        & 21           \\\bottomrule
  \end{tabular}
  \label{tab:undulating-crosses}
\end{table}

\subsubsection*{Missing inclusions}

This test, sketched in figure~\ref{fig:k-missing-inclusions}, echoes the example shown in figures~\ref{fig:cookies} and~\ref{fig:meso-mesh}: a square inclusion is present in sound cells and absent from faulty ones.
Therefore, $Y:=]0,1[^d$ and $K$ is expressed as~\eqref{eq:test-K-gen} with $K_1^Y = 1 + 99\chi$ and $K_2^Y = 1$, where $\chi\in C^0(Y,[0,1])$ is the continuous indicator function of the square inclusion, \myie{} $[(2-\sqrt{2})/4,(2+\sqrt{2})/4]^2$; the square's dimensions were chosen so as to have the same occupied surface in sound cells as the undulating fibres test case.

The results are shown in table~\ref{tab:missing-inclusions}.
The slight difference between this case and the previous one can only be ascribed to the change in conductivity pattern, since both are bidimensional at the mesoscopic scale.
Although this case has a higher conductivity contrast between sound and faulty cells than the previous one, it shows significant complexity reduction compared with the reference method.

For the sake of consistency, we retain only this conductivity pattern for the following experiments in sections~\ref{sec:comp-time-size}--\ref{sec:rank-precision}.

\begin{figure}
  \centering
  \begin{subfigure}{.4\linewidth}
    \centering
    \includegraphics[width=\textwidth]{missing-squares-cond.tikz}
    \caption{Conductivity}
    \label{fig:k-missing-inclusions}
  \end{subfigure}
  ~  
  \begin{subfigure}{.4\linewidth}
    \centering{}
    \includegraphics[width=\textwidth,height=\textwidth]{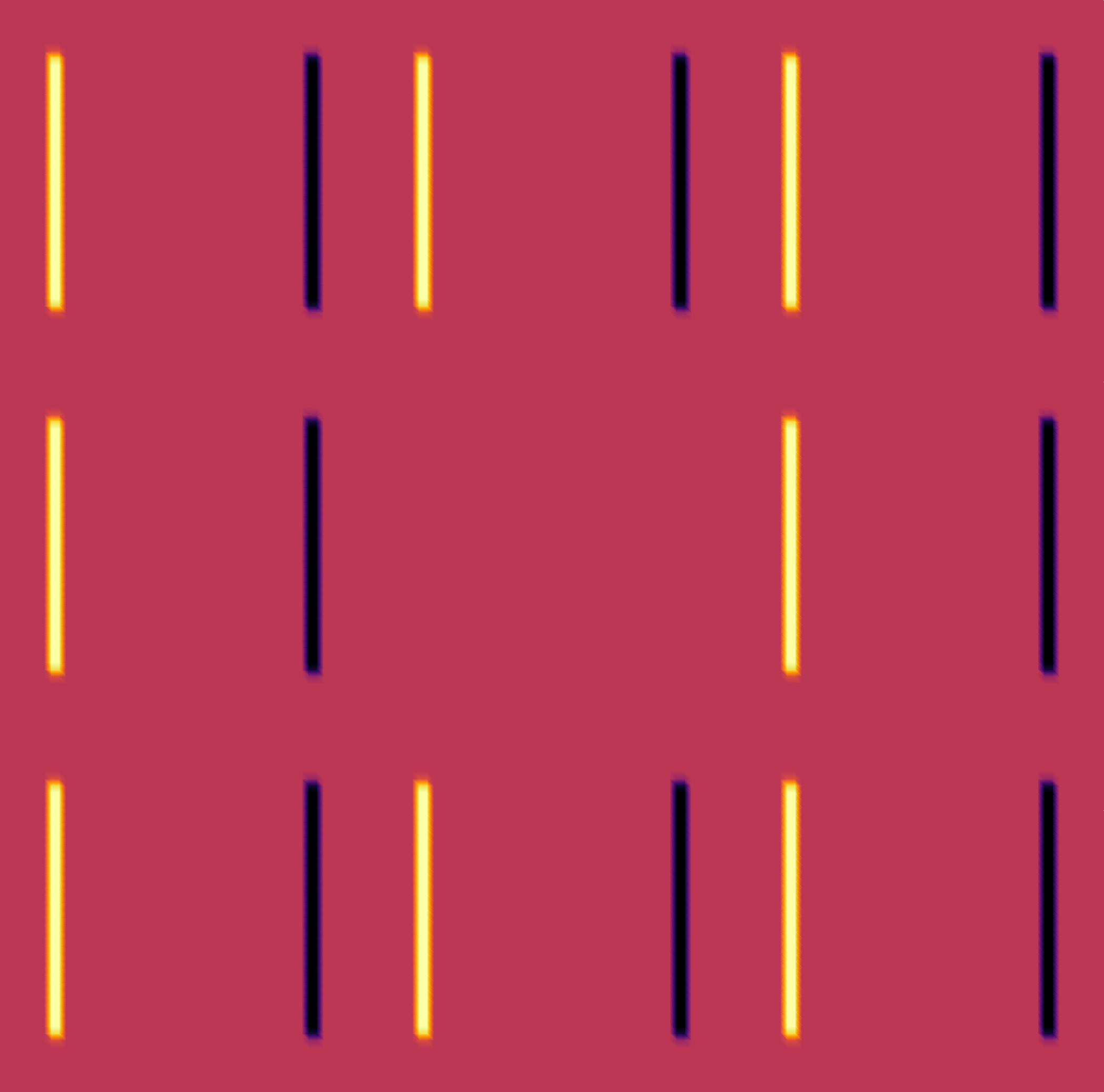}
    \caption{Source}
  \end{subfigure}

  \begin{subfigure}{.4\linewidth}
    \centering{}
    \includegraphics[width=\textwidth,height=\textwidth]{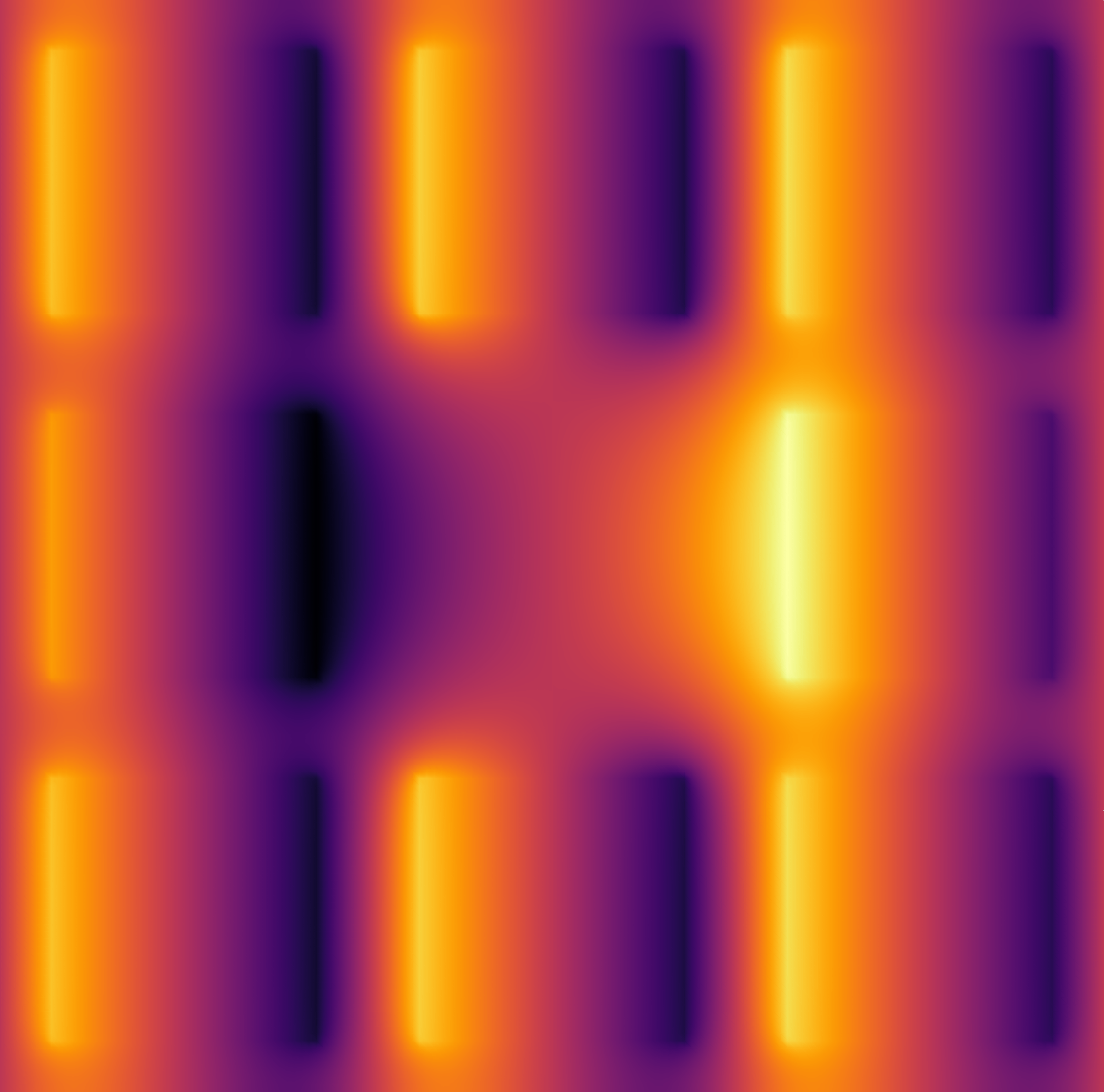}
    \caption{MsLRM approximation}
  \end{subfigure}
  ~
  \begin{subfigure}{.4\linewidth}
    \centering{}
    \includegraphics[width=\textwidth,height=\textwidth]{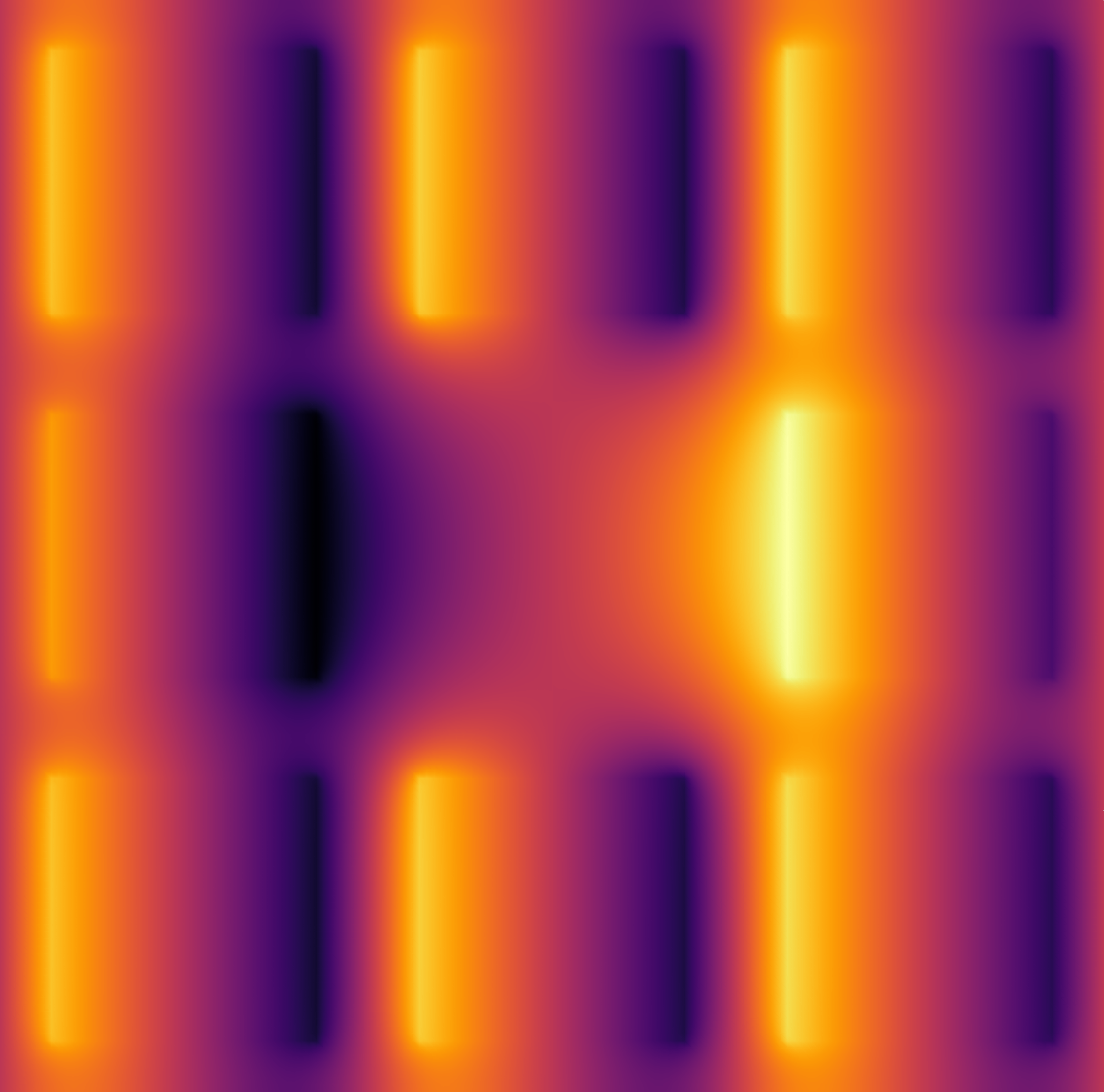}
    \caption{Reference solution}
  \end{subfigure}
  \caption{Missing inclusions test case: example with nine cells}
\end{figure}

\begin{table}
  \centering
  \caption{Missing inclusions test case}
  \begin{tabular}{*{4}{S}}\toprule
    {\multirow{2}*{$\#I$}} & {FEM}      & \multicolumn{2}{c}{MsLRM} \\\cmidrule(l){3-4}
                           & {time (s)} & {time (s)} & {rank}       \\\midrule
    25                     & 1.43       & 1.25       & 10           \\
    100                    & 9.46       & 3.6        & 17           \\
    225                    & 36.2       & 3.8        & 17           \\\bottomrule
  \end{tabular}
  \label{tab:missing-inclusions}
\end{table}

\subsection{Influence of domain size and source term}
\label{sec:comp-time-size}

The three initial tests of section~\ref{sec:cond-patterns} gave hints on the complexity reduction of the low-rank approximation method compared to a direct solution method.
To get a better insight into this, we observed the computational time of missing inclusions problems for a larger range of values of $\#I$, which resulted in figure~\ref{fig:time-cells-fem}.
The difference in complexity is made obvious.

These results were obtained with a quasi-periodic source term given by equation~\eqref{eq:src-corr1}.
We investigated the influence of the source term by running identical computations with two other source terms.
The first one is a uniform source term which smooths defects influence.
The second one is a centred peak given for $x \in D$ by
\begin{gather}
  \label{eq:src-peak}
  f(x) = e^{-10\norm{x-\theta}},
\end{gather}
where $\theta$ is the centre of $D$.
It is chosen so as to break periodicity.

As expected, figure~\ref{fig:time-cells_src} shows that with the uniform source term, the solution is quasi-periodic and the proposed method yields a high complexity reduction.
The peak source term problems have a higher complexity, as far as the MsLRM is concerned.
However, we see from figure~\ref{fig:rank-cells} that the approximation rank is bounded even in this latter case: there is still an underlying structure to the solution that allows an accurate approximation with low rank regardless of the domain size.

\begin{figure}
  \centering
  \includegraphics[width=\textwidth,axisratio=2]{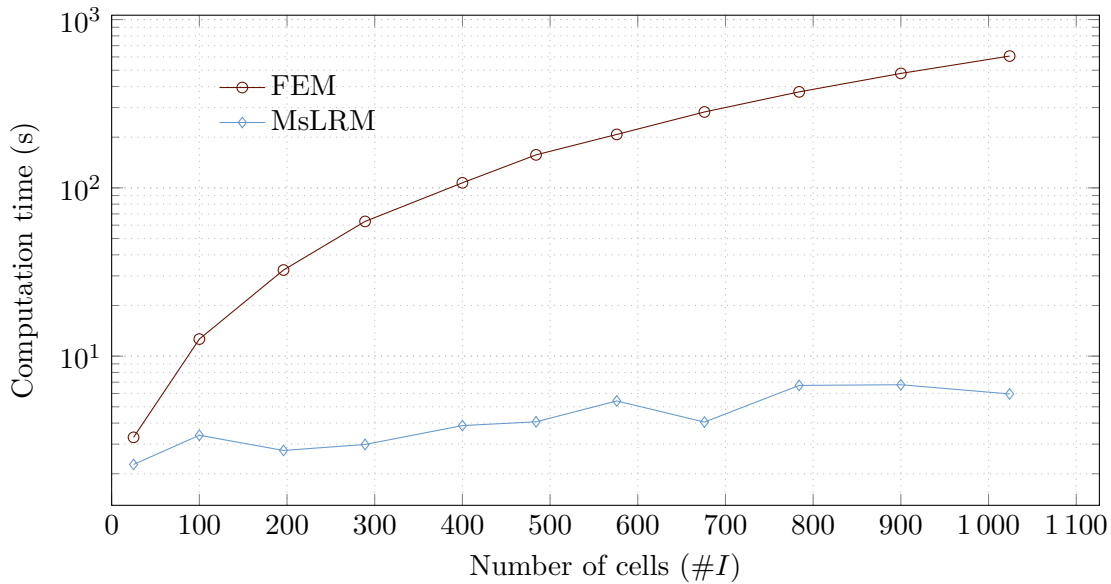}
  \caption{Domain size influence on MsLRM compared to FEM}
  \label{fig:time-cells-fem}
\end{figure}

\begin{figure}
  \centering{}
  \begin{subfigure}{\linewidth}
    \centering
    \includegraphics[width=.8\textwidth,axisratio=2]{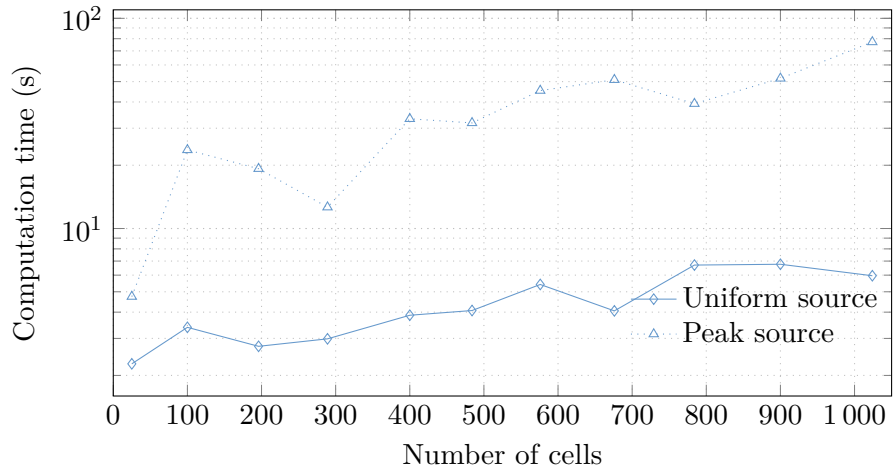}
    \caption{Impact on computational cost}
    \label{fig:time-cells_src}
  \end{subfigure}
  \\
  \begin{subfigure}{\linewidth}
    \centering 
    \includegraphics[width=.8\textwidth,axisratio=2]{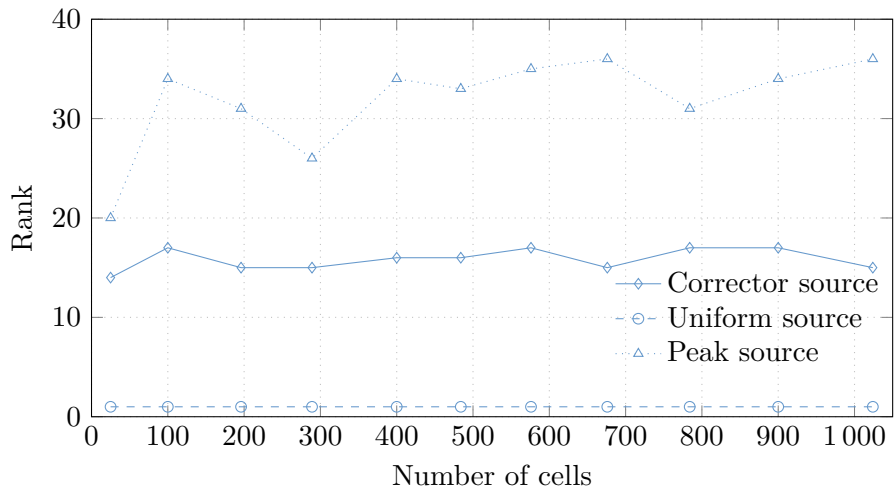}
    \caption{Impact on approximation rank}
    \label{fig:rank-cells}
  \end{subfigure}
  \caption{Source term influence on MsLRM}
\end{figure}

\subsection{Influence of the probability of defects}
\label{sec:proba-rank}

On the previous tests, the approximation rank as a function of the number of cells $\#I$ seemed to rapidly reach a plateau.
This was most obvious on figure~\ref{fig:rank-cells}.
One interpretation, illustrated on figure~\ref{fig:ranks-wrt-defects}, is that new patterns in the solution are caused by new configurations of defects, which increase the approximation rank for a given tolerance.
This plateau is a consequence of the medium's ergodicity: the larger the domain, the higher the probability to observe every possible configuration.
The number of cells before reaching the plateau depends on a number of parameters: the rank of the conductivity field (related to the number of cell types), the area of influence of a defect
(\mycf{} figures~\ref{fig:rank-wrt-defects_1}--\ref{fig:rank-wrt-defects_2} and \ref{fig:rank-wrt-defects_3}--\ref{fig:rank-wrt-defects_4}), and the probability of a defect.

\begin{figure}
  \centering
  \begin{subfigure}{.3\linewidth}
    \centering{}
    \includegraphics[width=\linewidth,height=\textwidth]{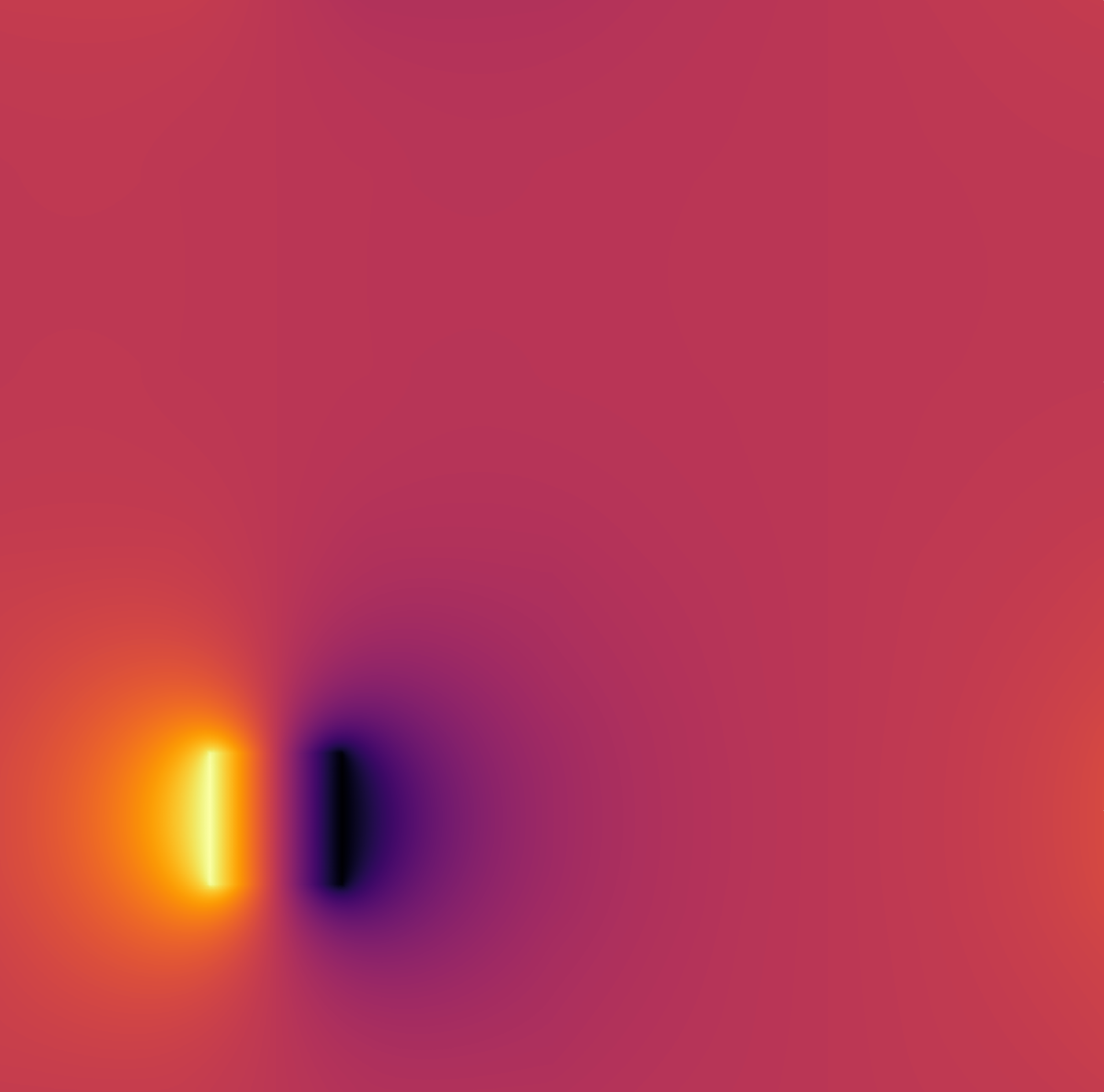}
    \caption{Rank 10}
    \label{fig:rank-wrt-defects_1}
  \end{subfigure}
  ~
  \begin{subfigure}{.3\linewidth}
    \centering{}
    \includegraphics[width=\linewidth,height=\textwidth]{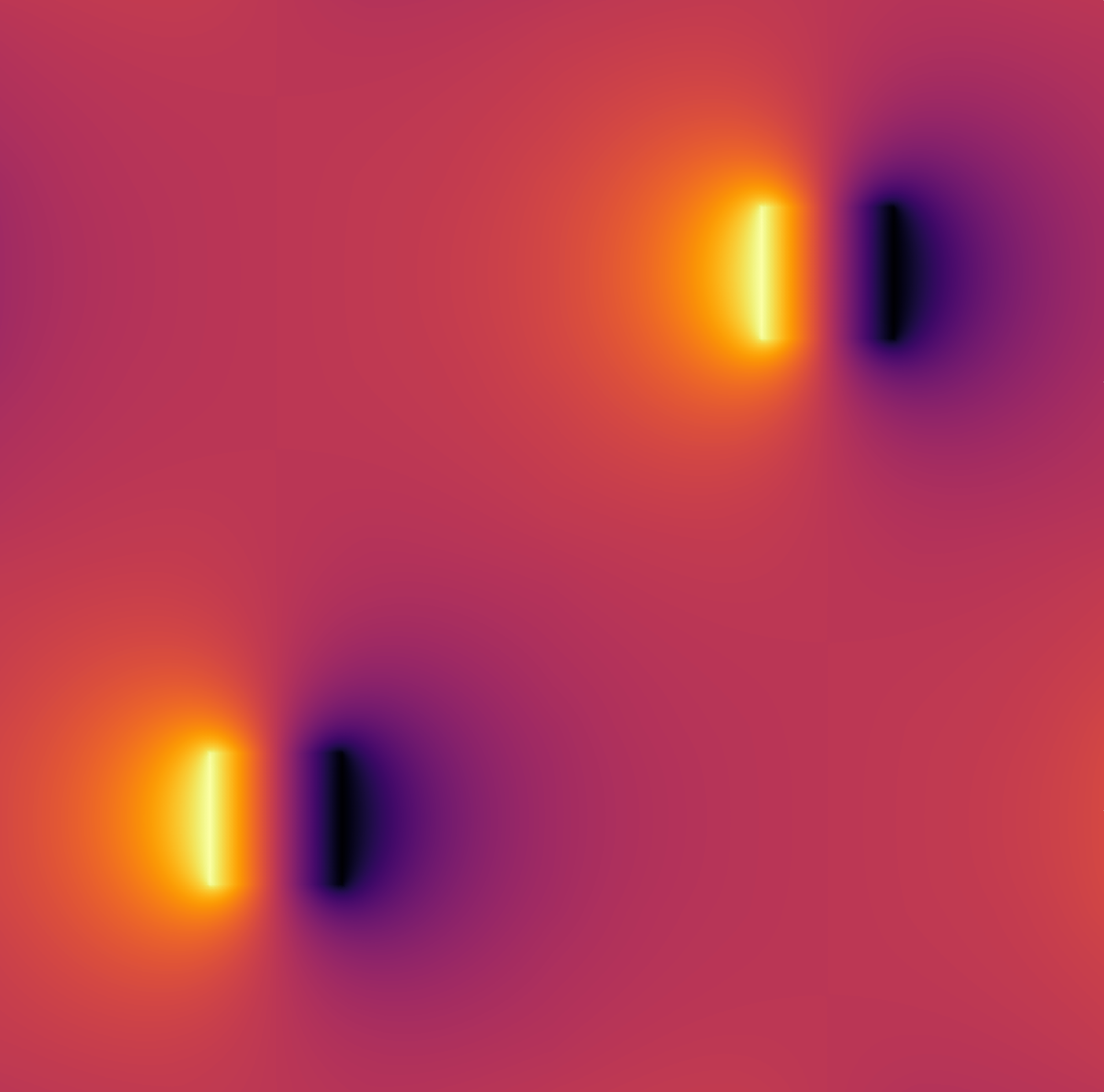}
    \caption{Rank 10}
    \label{fig:rank-wrt-defects_2}
  \end{subfigure}
  ~
  \begin{subfigure}{.3\linewidth}
    \centering{}
    \includegraphics[width=\linewidth,height=\textwidth]{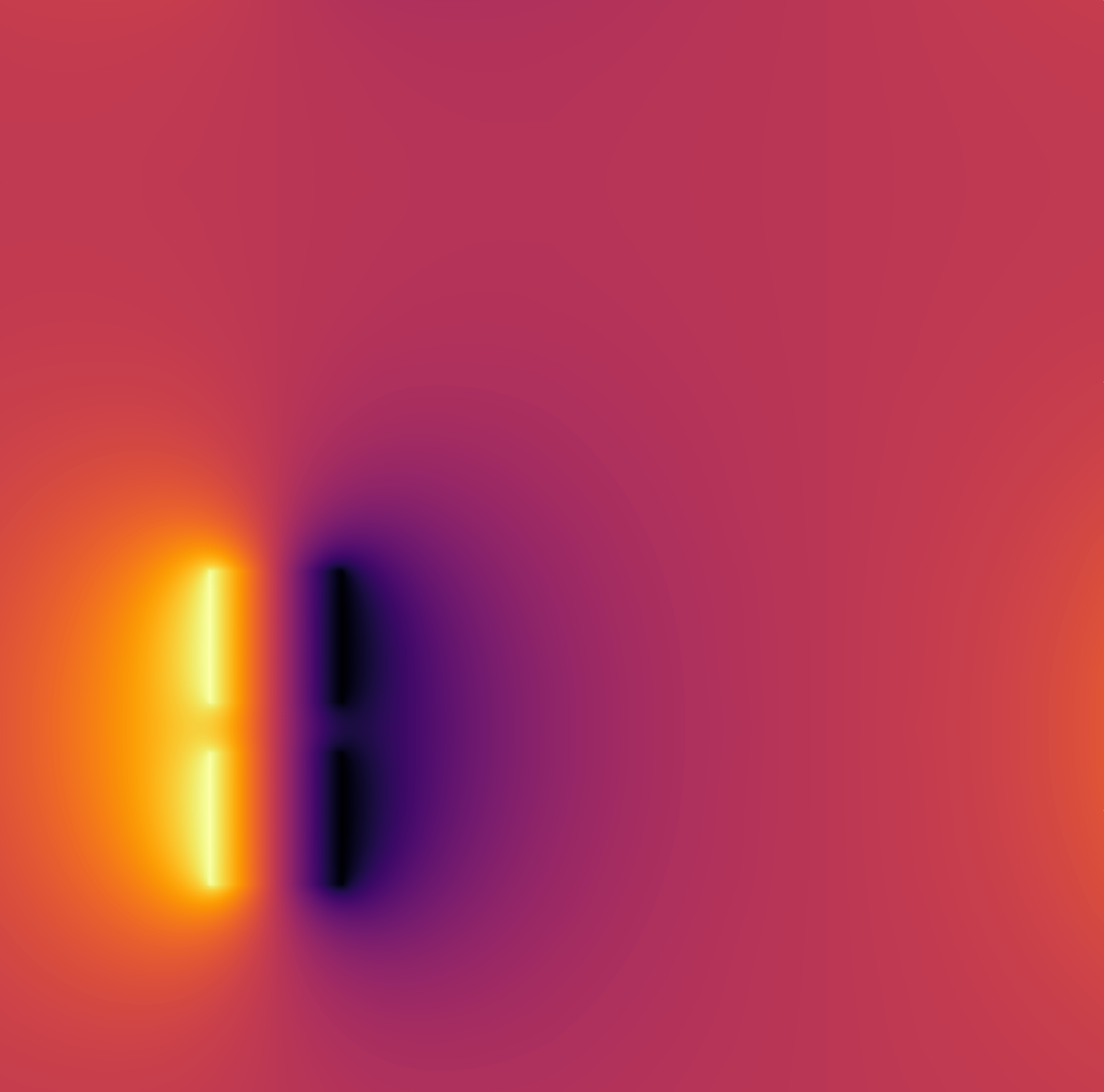}
    \caption{Rank 13}
    \label{fig:rank-wrt-defects_3}
  \end{subfigure}
  
  \begin{subfigure}{.3\linewidth}
    \centering{}
    \includegraphics[width=\linewidth,height=\textwidth]{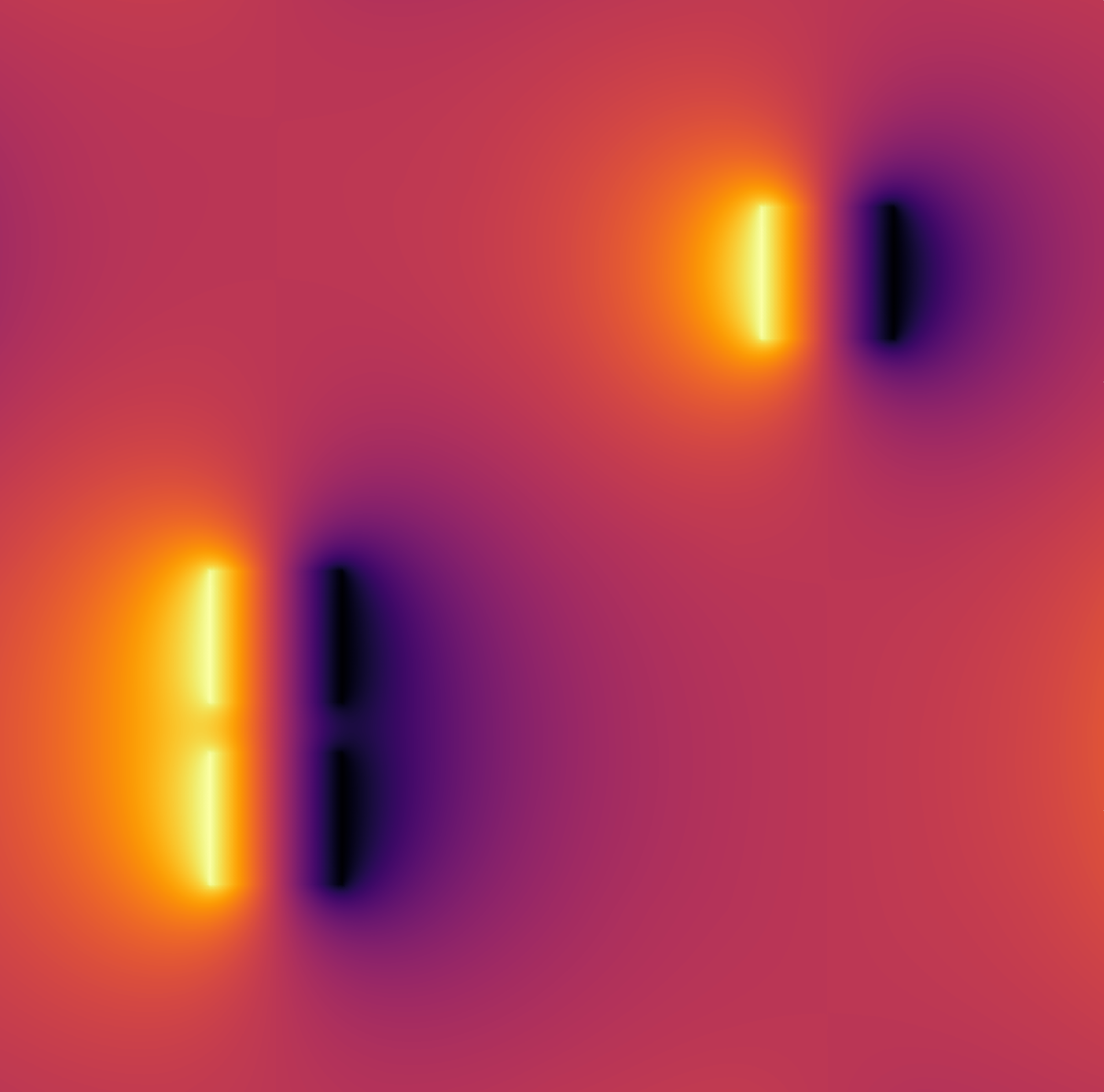}
    \caption{Rank 14}
    \label{fig:rank-wrt-defects_4}
  \end{subfigure}
  ~
  \begin{subfigure}{.3\linewidth}
    \centering{}
    \includegraphics[width=\linewidth,height=\textwidth]{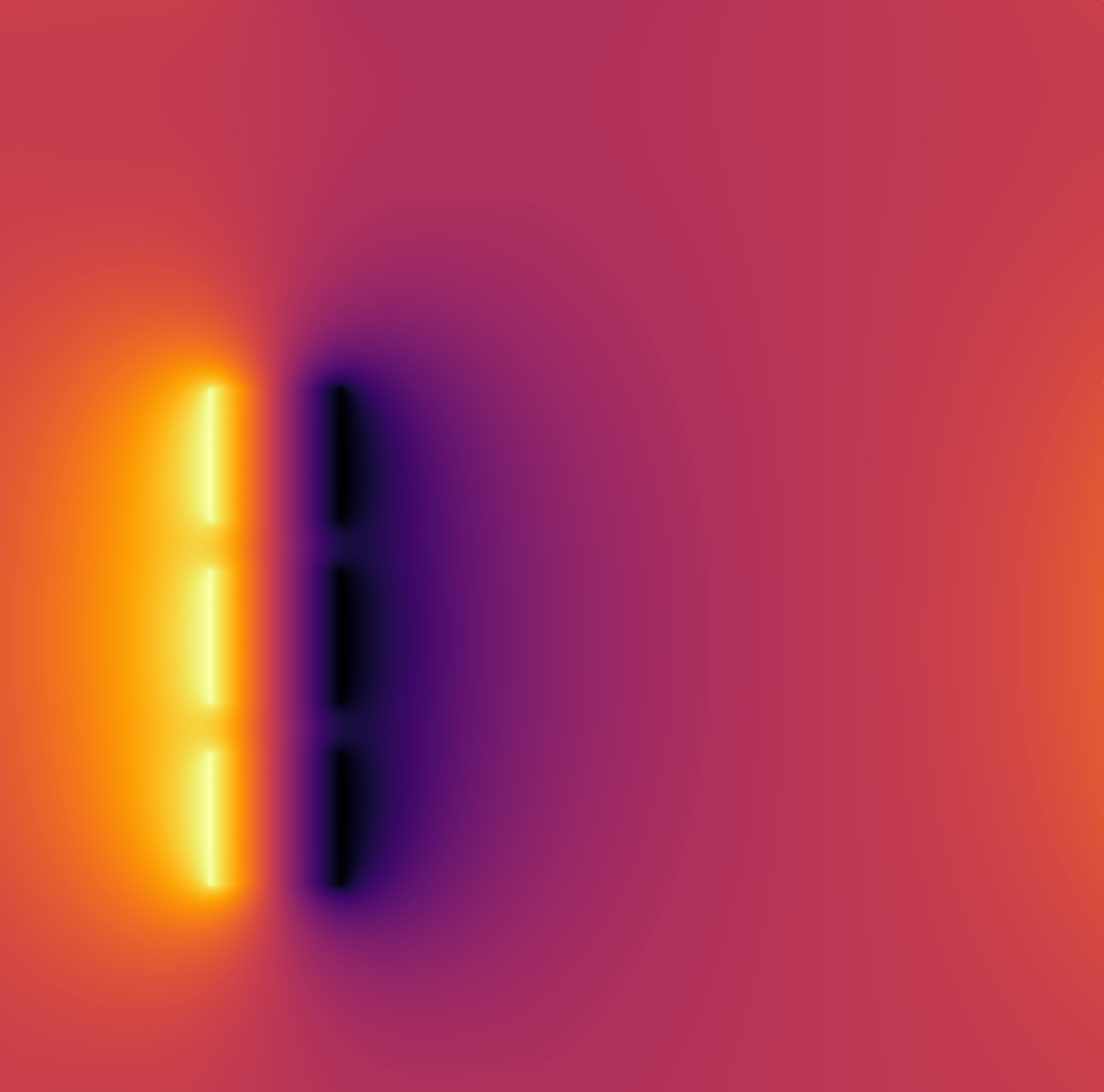}
    \caption{Rank 13}
    \label{fig:rank-wrt-defects_5}
  \end{subfigure}
  ~
  \begin{subfigure}{.3\linewidth}
    \centering{}
    \includegraphics[width=\linewidth,height=\textwidth]{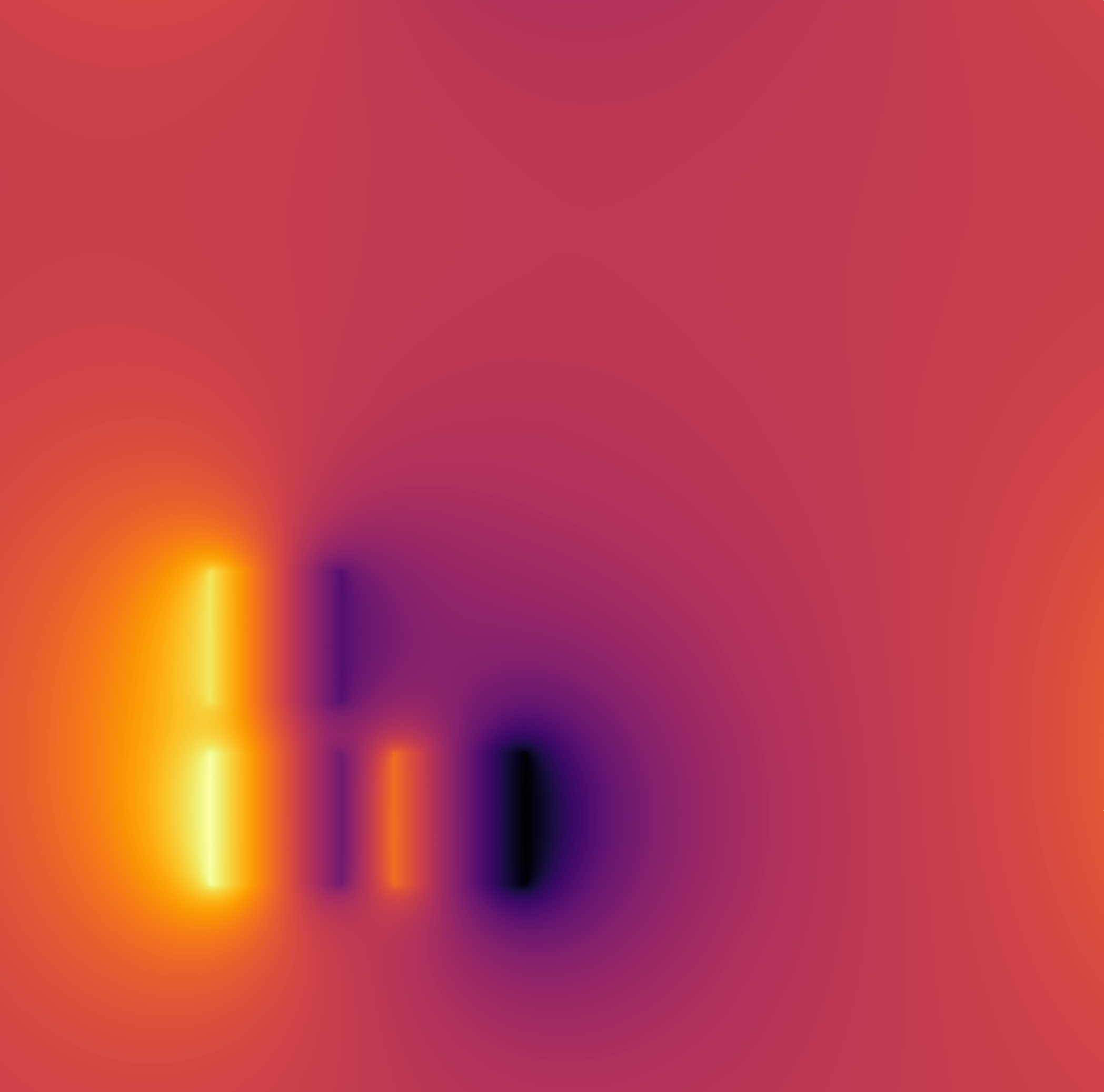}
    \caption{Rank 14}
    \label{fig:rank-wrt-defects_6}
  \end{subfigure}
  \caption{Approximation rank for various configurations of square inclusion defects}
  \label{fig:ranks-wrt-defects}
\end{figure}

Furthermore, we observed the influence of the probability of defect $p$ on the approximation rank.
For each value of $p$, we observed the average rank over \num{100} computations.
Here, the defect is a square inclusion, as in figure~\ref{fig:ranks-wrt-defects}.
The results are plotted on figure~\ref{fig:rank-proba_sq-avg} and display the expected low values when $p$ goes to $0$ or $1$, where we tend to a periodic medium.
The graph is slightly asymmetric: the highest approximation ranks were encountered when cells with inclusions were more likely.
A missing inclusion in a medium with periodic inclusions has less effect than an inclusion in a uniform medium.

To investigate the variability of ranks, we plotted their variance for each value of $p$ on figure~\ref{fig:rank-proba_sq-var}.
As for the average rank, this graph is slightly skewed.
The highest values are when $p$ goes to $0$ or $1$, \myie{} when the probability of getting a periodic medium and the probability of having at least one defect are of similar order.
This can be mostly explained by considering figure~\ref{fig:ranks-wrt-defects}: in this case the solution associated with a perfectly periodic medium would be of rank 1; one defect yields an approximation of rank\footnote{All ranks given here are for approximations with the tolerance value in table~\ref{tab:param-ref}.} 10 in figure~\ref{fig:rank-wrt-defects_1}; figures~\ref{fig:rank-wrt-defects_2}--\ref{fig:rank-wrt-defects_6} show that additional defects cause a much smaller increase in rank---none if no new pattern appears (\mycf{} figures~\ref{fig:rank-wrt-defects_3} and~\ref{fig:rank-wrt-defects_5}).
Therefore, the approximation rank reaches its highest variance for values of $p$ that makes a periodic medium as likely as a medium with at least one defect.

\begin{figure}
  \centering{}
  \begin{subfigure}{\textwidth}
    \centering
    \includegraphics[width=.7\textwidth,axisratio=2]{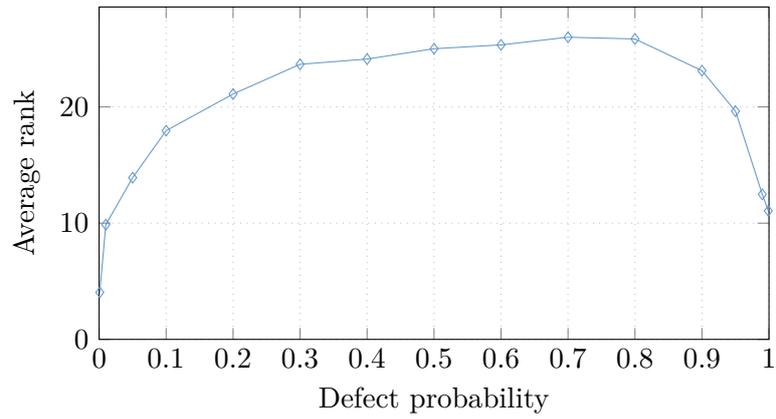}
    \caption{Average approximation rank}
    \label{fig:rank-proba_sq-avg}
  \end{subfigure}
  \\
  \begin{subfigure}{\textwidth}
    \centering 
    \includegraphics[width=.7\textwidth,axisratio=2]{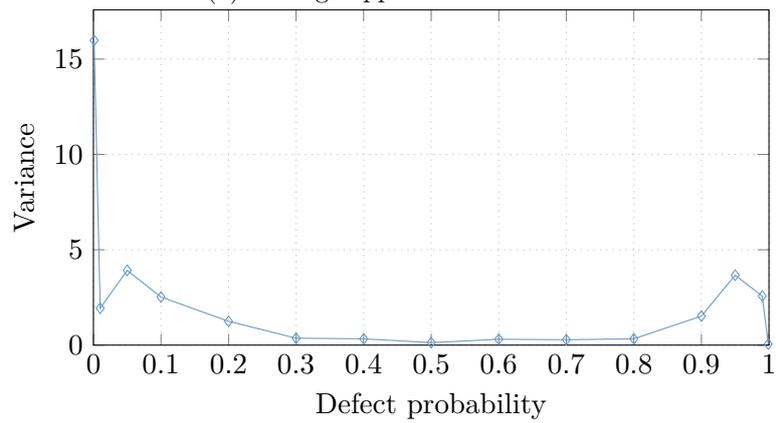}
    \caption{Variance of average approximation rank}
    \label{fig:rank-proba_sq-var}
  \end{subfigure}
  \caption{Effect of square inclusion probability on approximation rank ($\#I=400$, 100 samples)}
  \label{fig:rank-proba_sq}
\end{figure}

\subsection{Rank and precision}
\label{sec:rank-precision}

All previous results were obtained for a tolerance of $10^{-3}$.
We have seen the influence of conductivity patterns, problem size and source terms on the rank of the approximation.
Now, we analyse the convergence of the approximation with respect to the rank.
We consider a problem of missing inclusions as in section~\ref{sec:cond-patterns}, over a square domain of $400$ cells, and observe the evolution of the relative residual error as defined in equation~\eqref{eq:residual-error} with respect to the approximation rank.

Figure~\ref{fig:precision-rank} presents the results.
We observe an exponential convergence of the error with respect to the rank.
Tolerance remains a major factor in computational cost of the proposed low-rank method: for small domain size and high precision, a direct solution method would be more efficient.

\begin{figure}
  \centering 
  \includegraphics[width=.9\textwidth,axisratio=2]{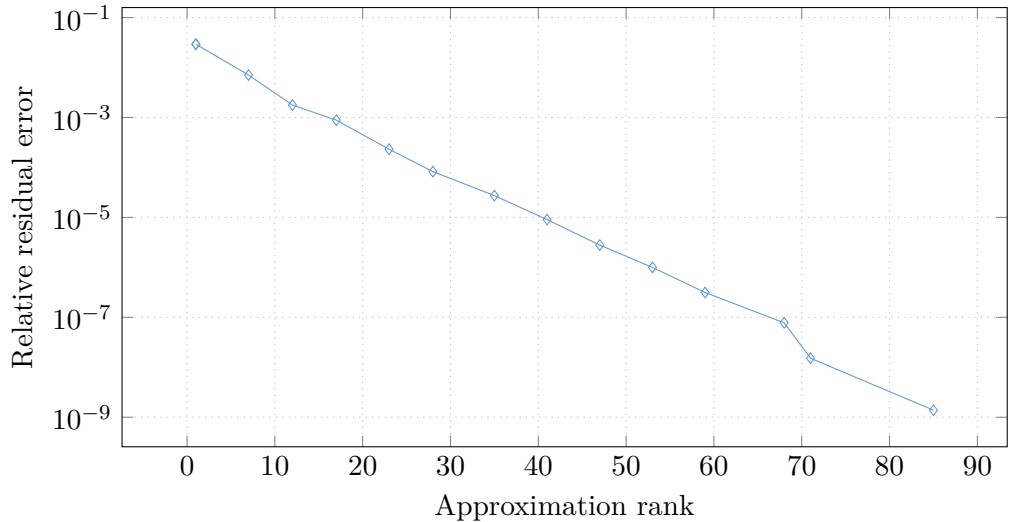}
  \caption{Approximation rank with respect to precision ($\#I=400$)}
  \label{fig:precision-rank}
\end{figure}

\section{Conclusion}
\label{sec:conclusion}

We have presented an approximation method to reduce the complexity of the solution of stationary diffusion problems in quasi-periodic media.
The method relies on a two-scale representation of the solution, which is identified with a tensor.
The method then exploits the fact that the solution admits accurate low-rank approximations.
A greedy algorithm is employed to build a non-optimal yet cost-efficient low-rank approximation with a desired precision.
The proposed method can be easily adapted to a larger class of linear elliptic PDEs.

Cost-efficiency has been illustrated comparatively to a direct solution method in numerical experiments with several conductivity patterns which are typical in composite materials.
Complexity reduction compared to the direct solution method has been observed on the different experiments.
Finally, the validity of the low-rank assumption has been tested with respect to precision and perturbation of periodicity.
A plateau in approximation rank with respect to domain size increase, attributed to the medium ergodicity, has been observed and suggests good performances for computations on large domains, even in case of low periodicity.

\section*{Acknowledgement}

The authors gratefully acknowledge the financial support from the Fondation CETIM.

\printbibliography

\end{document}